\documentclass[reqno]{amsart}
 \theoremstyle{definition}
 
 \theoremstyle{remark}

 \numberwithin{equation}{section}

\usepackage{graphicx}
\usepackage{amssymb, amsmath, amsthm}
\usepackage{amsfonts}
\usepackage{amssymb}
\usepackage{float}
\usepackage{mathrsfs}
\usepackage{hyperref}
\hypersetup{
    bookmarks=true,         
    pdffitwindow=false,     
    pdfstartview={FitH},    
    colorlinks=false,       
}
\usepackage{enumitem} 
\newtheorem{theorem}{Theorem}

\newtheorem{definition}[theorem]{Definition}

\newtheorem{lemma}[theorem]{Lemma}

\newtheorem{proposition}[theorem]{Proposition}

\newcommand{\dis}{\displaystyle}

\newcommand{\divv}{\text{\rm div}}

\newcommand{\dt}{\partial_t}

\newcommand{\R}{\mathbb R}

\newcommand{\ds}{\displaystyle}
\newcommand{\intox}{\int_{\R^3}}
\newcommand{\intoxt}{\int_0^t\int_{\R^3}}
\newcommand{\Int}{\int_0^t\int_{\R^3}}

\numberwithin{equation}{section}
\numberwithin{theorem}{section}

\numberwithin{figure}{section}

\begin{document}

%
%
%
%
%
%
%
%
%








\title[Navier-Stokes-Poisson Equations with finite energy]
 {Existence and a blow-up criterion of solution to the 3D compressible Navier-Stokes-Poisson equations with finite energy}
 
\author{Anthony Suen} 

\address{Department of Mathematics and Information Technology\\
The Education University of Hong Kong}

\email{acksuen@eduhk.hk}

\date{June 30, 2019}



\begin{abstract}
We study the low-energy solutions to the 3D compressible Navier-Stokes-Poisson equations. We first obtain the existence of smooth solutions with small $L^2$-norm and essentially bounded densities. No smallness assumption is imposed on the $H^4$-norm of the initial data. Using a compactness argument, we further obtain the existence of weak solutions which may have discontinuities across some hypersurfaces in $\R^3$. We also provide a blow-up criterion of solutions in terms of the $L^\infty$-norm of density.
\end{abstract}

\maketitle
\section{Introduction}\label{Introduction section}

In this present paper, we consider the following isentropic compressible Naiver-Stokes-Poisson (NSP) equations in the whole space $\R^3$ ($j=1,2,3$):
\begin{align}
\label{Navier Stokes Poisson} \left\{ \begin{array}{l}
\rho_t + \divv (\rho u) =0, \\
(\rho u^j)_t + \divv (\rho u^j u) + (P)_{x_j} = \mu\,\Delta u^j +
\lambda \, (\divv \,u)_{x_j}  + \rho \phi_{x_j},\\
\Delta\phi=\rho-\tilde\rho,
\end{array}\right.
\end{align}
Here $x\in\R^3$ is the spatial coordinate and $t\ge0$ stands for the time. The unknown functions $\rho=\rho(x,t)$, $u=(u^1,u^2,u^3)(x,t)$ and $\phi=\phi(x,t)$ represent the electron density, electron velocity and the electrostatic potential respectively. $P=P(\rho)$ is a function in $\rho$ which denotes the pressure and $\tilde\rho>0$ is a fixed constant, and $\mu$, $\lambda$ are positive viscosity constants. The system \eqref{Navier Stokes Poisson} is equipped with initial condition
\begin{equation}\label{initial condition}
(\rho(\cdot,0)-\tilde\rho,u(\cdot,0),\phi(\cdot,0)) = (\rho_0-\tilde\rho,u_0,\phi_0)
\end{equation}
with the compatibility condition on $\rho_0$, namely
\begin{equation}\label{compatibility condition}
\intox(\rho_0-\tilde\rho)=0.
\end{equation}

The NSP system \eqref{Navier Stokes Poisson} was used for describing the dynamics of a compressible fluid of electron in which the fluid interacts with its own electric field under the influence of a charged ion background at a given temperature. Equations \eqref{Navier Stokes Poisson}$_1$ and \eqref{Navier Stokes Poisson}$_2$ give the conservation of charge and conservation of momentum respectively, while equation \eqref{Navier Stokes Poisson}$_3$ is the self-consistent Poisson equation which relates the electron density and electrostatic potential. We refer to \cite{CG00}, \cite{Degond00}, \cite{D03}, \cite{DT08} for more detailed discussions. 

The system \eqref{Navier Stokes Poisson}-\eqref{initial condition} has been studied by various mathematicians and we first recall some known results from the literature. On the one hand, Li-Matsumura-Zhang \cite{lmz10} obtained the global existence of small-smooth type solutions using the method by Matsumura-Nishida \cite{mn79}-\cite{mn80} with smallness assumptions on the initial data. The authors in \cite{lmz10} further proved that the density $\rho$ converges to its equilibrium state in $L^2$ and $L^\infty$-norm with optimal rates of convergence. On the other hand, the global existence of large-weak type solutions of \eqref{Navier Stokes Poisson}-\eqref{initial condition} with large initial data was proved in Donatelli \cite{D03} and Zhang-Tan \cite{ZT07} using the theory of P. L. Lions \cite{lions98}.

In this present work, we try to deepen our understanding on the compressible NSP system by addressing the solutions to \eqref{Navier Stokes Poisson} from another new perspective, in the sense that the initial data \eqref{initial condition} is assumed to be small in some weaker norms ($L^2$) with nonnegative and essentially bounded initial densities, and {\it no} further smallness assumption is imposed on the higher-regularity norms of the initial data. Such idea was first initiated by Hoff \cite{hoff95}-\cite{hoff06} in studying compressible Navier-Stokes system which was later extended by Suen \cite{suen13}-\cite{suen14} for compressible Naiver-Stokes system with potential forces as well as by Suen-Hoff \cite{sh12} for compressible magnetohydrodynamics (MHD). The weak solutions obtained in the present work are known as the {\it intermediate} weak solutions which enjoy the following properties:
\begin{itemize}
\item The density $\rho$ and velocity gradient $\nabla u$ may exhibit discontinuities across some hypersurfaces in $\R^3$, such phenomenon cannot be observed by those small-smooth type classical solutions.
\item These intermediate weak solutions would have more regularity than the large-weak type solutions developed by Lions \cite{lions98}, so that the uniqueness and continuous dependence of solutions may still be obtained (see \cite{hoff06} for the case without external force). 
\end{itemize}
Furthermore, as a by-product of our analysis, we provide a blow-up criterion of smooth solution to \eqref{Navier Stokes Poisson}-\eqref{initial condition} in terms of density. Such result is parallel to those obtained in Navier-Stokes system \cite{swz11} as well as for compressible MHD system \cite{suen13a}. In the present work, we allow vacuum in the initial density $\rho_0$ and there is {\it no} smallness assumption imposed on the initial data in obtaining such blow-up criterion.

The main novelties of this current work can be summarised as follows:

\noindent{1.} We generalise the results obtained in \cite{lmz10} in the way that we obtain the existence of classical solution to \eqref{Navier Stokes Poisson}-\eqref{initial condition} {\it without} smallness assumption on the $H^4$ of the initial data. 

\noindent{2.} We prove the existence of intermediate weak solutions to \eqref{Navier Stokes Poisson}-\eqref{initial condition} which can be viewed as an extensions from those in \cite{hoff95}-\cite{hoff06} for compressible Navier-Stokes. 

\noindent{3.} We obtain a blow-up criterion for \eqref{Navier Stokes Poisson}-\eqref{initial condition} in terms of the $L^\infty$-norm on the density. Such result is parallel to the one for the compressible Navier-Stokes system as given in \cite{swz11}. 

\medskip

We provide a brief outline on the analysis and idea behind our work. We introduce the following auxiliary variable associated with the system \eqref{Navier Stokes Poisson} which is known as the {\it effective viscous flux} $F$. It is given by 
\begin{equation}\label{def of effective viscous flux}
F=(\mu+\lambda)\divv(u)-(P(\rho)-P(\tilde\rho)).
\end{equation}
The variable $F$ has been studied extensively by Hoff in \cite{hoff95}-\cite{hoff06}, and we refer to those references for more detailed discussions on $F$. Upon rearranging terms, the momentum equation \eqref{Navier Stokes Poisson}$_2$ can be rewritten in terms of $F$:
\begin{equation}\label{momentum eqn rewritten}
\rho\dot u^j=F_{x_j}+\mu\omega^{j,k}_{x_k}-\rho\phi_{x_j}.
\end{equation}
Differentiate \eqref{momentum eqn rewritten} with respect to $x_j$ and sum over $j$, we obtain the following equation for $F$:
\begin{equation}\label{Poisson eqn of F}
\Delta (F)=\divv(\rho\dot u-\rho\nabla\phi).
\end{equation}
Equations \eqref{momentum eqn rewritten}-\eqref{Poisson eqn of F} will be crucial in obtaining {\it a priori} bounds on the solutions, and the importance can be explained heuristically as follows:

\noindent{1.} In estimating higher-regularity norms (for example $\int_{0}^{1}\!\!\!\int_{\R^3}t^3|\nabla u|^4dxdt$) of the weak solutions, one {\it cannot} merely applying the embedding $W^{2,2}\hookrightarrow W^{1,4}$ as $\nabla u$ may be discontinuous across hypersrufaces of $\R^3$ (see for example \cite{hoff95} for more details). With the help of $F$ and $\omega$, we are able to observe the following decomposition of $\Delta u$:
\begin{equation*}
u^j_{x_k x_k}=\omega^{j,k}_{x_j}+\left[(\mu+\lambda)^{-1}F\right]_{x_j}+\left[(\mu+\lambda)^{-1}(P-P(\tilde\rho))\right]_{x_j}.
\end{equation*}
If we anticipate that $F(\cdot,t),\omega(\cdot,t)\in H^1$ and $P(\cdot,t)-P(\tilde\rho)\in L^2\cap L^\infty$, the term $u^j_{x_k x_k}$ should then be in $W^{-1,4}$, and hence the desired bound for $\nabla u$ in $L^4$ follows by a Fourier-type multiplier theorem.

\noindent{2.} Another important application of the equations \eqref{momentum eqn rewritten}-\eqref{Poisson eqn of F} can be revealed in studying the pointwise bounds on the density. With the help of the effective viscous flux $F$, we can rewrite equation \eqref{Navier Stokes Poisson}$_1$ as follows:
\begin{align*}
(\mu+\lambda)\frac{d}{dt}[\log\rho(x(t),t)]+P(\rho(x(t),t))-P(\tilde\rho)=-\tilde\rho F(x(t),t),
\end{align*}
where $(x(t),t)$ is a particle path governed by $u$. Upon integrating with respect to time, we observe that the oscillation in density can be controlled by the time integral of $-\tilde\rho F(x(t),t)$. By utilizing the Poisson's equation \eqref{Poisson eqn of F} and the claimed {\it a priori} bounds on $F$, we are able to show that such time integral is bounded by the initial energy of the system which is taken to be small by our assumption. Hence the density remains bounded above in $L^\infty$ as compare to itself initially. 

\medskip

We now give a precise formulation of our results. We first define the system parameters $P$, $\mu$, $\lambda$ as follows. For the pressure function $P=P(\rho)$, we assume that
\begin{equation}\label{assumption on P}
\mbox{$P(\rho)\in C^2((0,\infty))$ with $P(0)=0$; $P(\rho)>0$, $P'(\rho)>0$ for $\rho>0$;}
\end{equation}
For the diffusion coefficients $\mu$ and $\lambda$, we assume that
\begin{equation}\label{assumption on viscosity}
\mbox{$\mu>0$ and $0<\lambda<\frac{5\mu}{4}$.}
\end{equation}
It follows that
\begin{equation}\label{assumption on viscosity in q}
\frac{\mu}{\lambda}>\frac{(q-2)^2}{4(q-1)}
\end{equation}
for $q=6$ and consequently for some $q>6$, which we now fix. We also remark that the above conditions \eqref{assumption on P}-\eqref{assumption on viscosity in q} as imposed on $P$, $\mu$ and $\lambda$ are consistent with those used by Li-Matsumura \cite{lm11} and Hoff \cite{hoff05} for compressible Navier-Stokes system. Condition \eqref{assumption on P} is considered to be more general than those used in Hoff \cite{hoff95} which includes the special case $P(\rho)=K\rho^\gamma$ for $\gamma\ge1$ and $K>0$. The assumptions \eqref{assumption on viscosity}-\eqref{assumption on viscosity in q} are required for technical reasons in obtaining {\it a priori} estimates and will be particularly used in proving Lemma~\ref{Lq bound on u lemma} (notice that \eqref{assumption on viscosity}-\eqref{assumption on viscosity in q} are consistent with those given in \cite{hoff95}).

Next we state the assumptions on the initial data $(\rho_0,u_0,\phi_0)$. We assume there is a positive number $N$, which may be arbitrarily large such that
\begin{equation}\label{Lq bound on initial data}
\|u_0\|_{L^{q}}\le N
\end{equation}
where $q$ is defined in \eqref{assumption on viscosity in q}. From now on, for $\rho_0-\tilde\rho, u_0,\nabla\phi_0\in L^2(\R^3)$, we also write
\hfill
\begin{equation}\label{def of L^2 initial data}
C_{0}=\|\rho_0-\tilde\rho\|^2_{L^2}+\|u_0\|^2_{L^{2}}+\|\nabla\phi_0\|^2_{L^2}
\end{equation}
for the sake of convenience without further referring.

Weak solutions to the system \eqref{Navier Stokes Poisson}-\eqref{initial condition} can be defined as follows. Given $T>0$, we say that $(\rho-\tilde\rho,u,\phi)$ is a {\it weak solution} of \eqref{Navier Stokes Poisson}-\eqref{initial condition} if 
\begin{itemize}
\item $(\rho-\tilde\rho,\,\rho u,\nabla\phi)\in C([0,T];H^{-1}(\R^3))$;
\item $(\rho-\tilde\rho,u,\phi)|_{t=0}=(\rho_0,u_0,\phi_0)$;
\item $\nabla u\in L^2(\R^3\times(0,T))$;
\end{itemize}
and the following integral identities \eqref{weak form of mass equation}-\eqref{weak form of poisson equation} hold for all $t_1,t_2\in[0,T]$ and $C^1$ test functions $\varphi$ having uniformly bounded support in $x$ for $t\in[t_1,t_2]$:
\begin{align}\label{weak form of mass equation}
\left.\int_{\R^3}\rho(x,\cdot)\varphi(x,\cdot)dx\right|_{t_1}^{t_2}=\int_{t_1}^{t_2}\!\!\!\int_{\R^3}(\rho\varphi_t + \rho u\cdot\nabla\varphi)dxdt;
\end{align}
\begin{align}\label{weak form of momentum equation}
\left.\int_{\R^3}(\rho u^{j})(x,\cdot)\varphi(x,\cdot)dx\right|_{t_1}^{t_2}=&\int_{t_1}^{t_2}\!\!\!\int_{\R^3}[\rho u^{j}\varphi_t + \rho u^{j}u\cdot\nabla\varphi + P(\rho)\varphi_{x_j}]dxdt\notag\\
& - \int_{t_1}^{t_2}\!\!\!\int_{\R^3}[\mu\nabla u^{j}\cdot\nabla\varphi + (\mu - \xi)(\divv\,u)\varphi_{x_j}]dxdt\\
&+ \int_{t_1}^{t_2}\!\!\!\int_{\R^3}\rho\phi_{x_j}\varphi dxdt;\notag
\end{align}
\begin{align}\label{weak form of poisson equation}
-\int_{t_1}^{t_2}\!\!\!\int_{\R^3}\nabla\phi\cdot\nabla\varphi dxdt=\int_{t_1}^{t_2}\!\!\!\int_{\R^3}(\rho-\tilde\rho)\varphi dxdt.
\end{align}

We adopt the following usual notations for H\"older seminorms: for $v:\R^3\to \R^m$ and $\alpha \in (0,1]$, 
\hfill
$$\langle v\rangle^\alpha = \sup_{{x_1,x_2\in 
\R^3}\atop{x_1\not=x_2}}
{{|v(x_2) -v(x_1)|}\over{|x_2-x_1|^\alpha}}\,;$$
and for $v:Q\subseteq\R^3 \times[0,\infty)\to \R^m$ and $\alpha_1,\alpha_2 \in (0,1]$,
\hfill
$$\langle v\rangle^{\alpha_1,\alpha_2}_{Q} = \sup_{{(x_1,t_1),(x_2,t_2)\in 
Q}\atop{(x_1,t_1)\not=(x_2,t_2)}}
{{|v(x_2,t_2) - v(x_1,t_1)|}\over{|x_2-x_1|^{\alpha_1} + |t_2-t_1|^{\alpha_2}}}\,.$$

We denote the material derivative of a given function $v$ by
$\dot{v}=v_t + \nabla v\cdot u$,
Finally if $I\subset [0,\infty)$ is an interval, $C^1(I;X)$ will be the elements $v\in C(I;X)$ such that the distribution derivative $v_t\in {\mathcal D}'(\R^3\times{\rm int}\,I)$ is an element of $C(I;X)$.

We make use of the following standard facts (refer to Ziemer \cite{ziemer89} for details). First, given $r\in[2,6]$ there is a constant $C(r)$ such that for $w\in H^1 (\R^3)$,
\begin{equation}\label{sobolev bound r}
\|w\|_{L^r} \le C(r) \left(\|w\|_{L^2}^{(6-r)/2r}\|\nabla w\|_{L^2}^{(3r-6)/2r}\right).
\end{equation}
Next, for any $r\in (3,\infty)$ there is a constant $C(r)$ such that for $w\in W^{1,r}(\R^3)$,
\begin{equation}\label{infty bound r}
\|w\|_{L^\infty} \le C(r) \|w\|_{W^{1,r}}
\end{equation}
and
\begin{equation}\label{holder bound r}
\langle w\rangle^\alpha_{\R^3}\le C(r)\|\nabla w\|_{L^r},
\end{equation}
where $\alpha=1-3/r$. If $\Gamma$ is the fundamental solution of the Laplace operator on $\R^3$, then there is a constant $C$ such that for any $g\in L^2(\R^3)\cap L^4(\R^3)$, 
\begin{equation}\label{bound on green function}
\|\Gamma_{x_j}*g\|_{L^\infty} \le C\left(\|g\|_{L^2} + \|g\|_{L^4}\right).
\end{equation}

\medskip

The following are the main results of this paper. First of all, Theorem~\ref{existence theorem} shows that given $T>0$, under a smallness assumption on the $L^2$-norm of the initial data, the smooth classical solution to \eqref{Navier Stokes Poisson} exists on $[0,T]$.
\begin{theorem}\label{existence theorem}
Let the system parameters $P$, $\mu$, $\lambda$ be given and satisfy the conditions \eqref{assumption on P}-\eqref{assumption on viscosity in q}. Given $N$, $d$, $\bar\rho$, $\tilde\rho>0$ and $q>6$, for each $T>0$, there exists $\delta_T>0$ such that if the initial data $(\rho_0-\tilde\rho,u_0,\nabla\phi_0)\in H^4$ satisfies \eqref{compatibility condition} and \eqref{Lq bound on initial data}-\eqref{def of L^2 initial data} and
\begin{equation}\label{smallness on initial energy depend on T}
C_0\le\delta_T
\end{equation}
\begin{equation}\label{pointwise bound on initial rho existence}
0\le {\rm ess}\inf\rho_0 \le {\rm ess}\sup\rho_0<\bar\rho-d,
\end{equation}
then the classical solution $(\rho-\tilde\rho,u,\phi)$ of \eqref{Navier Stokes Poisson}-\eqref{initial condition} exists on $\R^3\times[0,T]$.
\end{theorem}

Next in Theorem~\ref{global existence theorem}, we show that for a given $L^2$ initial data, under the smallness assumption \eqref{smallness on initial energy depend on T} on $C_0$, there exists a weak solution $(\rho-\tilde\rho,u,\phi)$ of \eqref{Navier Stokes Poisson}-\eqref{initial condition} defined on $\R^3\times[0,T]$ for each $T>0$.
\begin{theorem}\label{global existence theorem}
Let the system parameters $P$, $\mu$, $\lambda$ be given and satisfy the conditions \eqref{assumption on P}-\eqref{assumption on viscosity in q}. Given $N$, $d$, $\bar\rho$, $\tilde\rho>0$ and $q>6$, for each $T>0$, there are constants $C_T,\theta_T,\delta_T>0$ such that if the initial data $(\rho_0-\tilde\rho,u_0,\nabla\phi_0)\in L^2$ satisfies \eqref{compatibility condition} and \eqref{Lq bound on initial data}-\eqref{def of L^2 initial data} and
\begin{equation}\label{smallness on initial energy depend on T weak}
C_0\le\delta_T
\end{equation}
\begin{equation}\label{pointwise bound on initial rho existence weak}
0\le {\rm ess}\inf\rho_0 \le {\rm ess}\sup\rho_0<\bar\rho-d,
\end{equation}
then a weak solution $(\rho-\tilde\rho,u,\phi)$ of \eqref{Navier Stokes Poisson}-\eqref{initial condition} in the sense of \eqref{weak form of mass equation}-\eqref{weak form of poisson equation} exists on $\R^3\times[0,T]$. In particular, $(\rho-\tilde\rho,u,\phi)$ satisfies the following:
\begin{equation}\label{1.4.1}
\rho-\tilde\rho,\,\nabla\phi,\,\rho u\in C([0,T];H^{-1}(\R^3)),
\end{equation}
\begin{equation}\label{L2 of nabla u weak solution}
\nabla u\in L^2(\R^3\times[0,T]),
\end{equation}
\begin{equation}\label{H1 of u weak solution}
u(\cdot,t)\in H^1 (\R^3),\;t\in(0,T],
\end{equation}
\begin{equation}\label{H1 of F and omega weak solution}
\mbox{$\omega(\cdot,t),F(\cdot,t)\in H^1 (\R^3)$, $t\in[0,T]$,}
\end{equation}
\begin{equation}\label{holder norm of u weak solution}
\langle u\rangle^{\frac{1}{2},\frac{1}{4}}_{\R^3 \times [\tau,T]}\leq C(\tau)C_{0}^{\theta},\;\tau\in(0,T],
\end{equation}
where $C(\tau)$ may depend additionally on a positive lower bound for $\tau$,
\begin{equation}\label{poinwise bound on rho weak solution}
\mbox{$0\le\rho(x,t)\le\bar\rho$ a.e. on $\R^3\times[0,T]$},
\end{equation}
and
\begin{align}\label{enegry bound on weak solution}
\sup_{0\le t\le T}\int_{\R^3}\big[(\rho - \tilde\rho)^2 + \rho|u|^2 &+ |\nabla\phi|^2+\sigma|\nabla u|^2 + \sigma^3 ( F^2 + |\nabla\omega|^2 )\big]dx\notag\\
+\int_{0}^{T}\!\!\!\int_{\R^3}\big[|\nabla u|^2 +& \sigma(\rho|\dot u|^2 + |\nabla\omega|^2 )+\sigma^{3}|\nabla\dot{u}|^2\big]dxds\le C_TC_{0}^{\theta_T}&
\end{align}
where $\sigma(t)=\min\{1,t\}$. 
\end{theorem}

Finally in Theorem~\ref{blow up theorem}, we obtain a blow-up criterion for the classical solutions to \eqref{Navier Stokes Poisson}-\eqref{initial condition} for the isothermal case {\it without} any smallness assumption on the initial data.
\begin{theorem}\label{blow up theorem}
Let the system parameters $P$, $\mu$, $\lambda$ be given and satisfy the conditions \eqref{assumption on viscosity}-\eqref{assumption on viscosity in q} and
\begin{equation}\label{isothermal assumption on P}
P(\rho)=K\rho,
\end{equation}
where $K>0$ is a given constant. Given $\tilde\rho>0$ and $(\rho_0-\tilde\rho,u_0,\nabla\phi_0)\in H^4(\R^3)$, assume that $(\rho-\tilde\rho,u,\phi)$ is the smooth classical solution on $\R^3\times[0,T]$. Let $T^*\ge T$ be the maximal existence time of the solution. If $T^*<\infty$, then we have
\begin{equation}
\lim_{t\to T^*}\|\rho\|_{L^\infty(\R^3\times(0,t))}=\infty.
\end{equation}
\end{theorem}

\medskip

The rest of the paper is organised as follows. In Section~\ref{Energy Estimates section}, we derive {\em a priori} bounds for smooth, local-in-time solutions to \eqref{Navier Stokes Poisson}-\eqref{initial condition} under the assumption that densities are non-negative and bounded. In Section~\ref{pointwise bound density section} we derive pointwise bounds for the density, bounds which are independent both of time and of initial smoothness. This will then close the estimates as obtained in Section~\ref{Energy Estimates section} to give an uncontingent estimate for the smooth solutions, thereby proving Theorem~\ref{existence theorem}. In Section~\ref{weak solution section}, together with the {\em a priori} bounds obtained in previous sections, we prove Theorem~\ref{global existence theorem} by applying compactness arguments. Finally in Section~\ref{blow-up section}, we give the blow-up criterion for solutions to \eqref{Navier Stokes Poisson}-\eqref{initial condition} by obtaining estimates on solutions without smallness assumption on initial data.

\section{Energy Estimates}\label{Energy Estimates section}

In this section we derive {\em a priori} bounds for smooth, local-in-time solutions of \eqref{Navier Stokes Poisson}-\eqref{initial condition} whose densities are non-negative and bounded. We first recall the following local-in-time existence theorem for the system \eqref{Navier Stokes Poisson}-\eqref{initial condition} (see for example \cite{lmz10} and the references therein): 

\begin{theorem}\label{local existence theorem}
Given $\tilde\rho>0$ and initial data $(\rho_0-\tilde\rho,u_0,\phi_0)\in H^4(\R^3)$, we can find $T>0$ such that the classical solution $(\rho-\tilde\rho,u,\phi)$ to \eqref{Navier Stokes Poisson}-\eqref{initial condition} exists on $\R^3\times[0,T]$. Moreover, $(\rho-\tilde\rho,u,\phi)$ satisfies
\begin{equation}\label{classical rho}
\rho - \tilde{\rho}\in C^0([0,T];H^4(\R^3))\cap C^1 ([0,T];H^3(\R^3)),
\end{equation}
\begin{equation}\label{classical u}
u\in C^0([0,T];H^4(\R^3))\cap C^1([0,T];H^2 (\R^3)),
\end{equation}
and
\begin{equation}\label{classical phi}
\nabla\phi\in C^0([0,T];H^5(\R^3)).
\end{equation} 
\end{theorem}

Given $T>0$, we now fix $\tilde\rho,\bar{\rho}$ as described in Section~\ref{Introduction section} and a smooth classical solution $(\rho-\tilde\rho,u,\phi)$ of \eqref{Navier Stokes Poisson}-\eqref{initial condition} on $\R^3\times[0,T]$ satisfying \eqref{classical rho}-\eqref{classical phi} with initial data $(\rho_0-\tilde\rho,u_0,\nabla\phi_0)$. With respect to $(\rho-\tilde\rho,u,\phi)$ , we then define functionals $\Phi(t)$ and $H(t)$ for a given such solution by
\begin{align}\label{def of Phi}
\Phi(t)&=\sup_{0\le s\le t}\Big[\|\rho^\frac{1}{2} u(\cdot,s)\|_{L^2}^2+||(\rho-\tilde\rho)(\cdot,s)||_{L^2}^2+\|\nabla\phi(\cdot,s)\|^2_{L^2}\Big]\notag\\
&\,\,\,+\sup_{0\le s\le t}\Big[\sigma(s)\|\nabla u(\cdot,s)\|^2_{L^2}+\sigma^3(s)\|\rho^\frac{1}{2}\dot{u}(\cdot,s)\|^2_{L^2}\Big]\notag\\
&\,\,\,+\int_{0}^{t}\Big[\|\nabla u(\cdot,s)\|_{L^2}^2+\sigma(s)\|\rho^\frac{1}{2}\dot{u}(\cdot,s)\|_{L^2}^2+\sigma^3(s)\|\nabla\dot{u}(\cdot,s)\|^2_{L^2}\Big]ds.
\end{align}
\begin{align}\label{def of H}
H(t)=\int_0^t\sigma^\frac{3}{2}(s)\|\nabla u(\cdot,s)\|^3_{L^3} ds&+\int_0^t\sigma^3(s)\|\nabla u(\cdot,s)\|^4_{L^4} ds\\
&+\Big |\sum_{1\le k_i,j_m\le 3}\int_{0}^{t}\!\!\!\int_{\R^3}\sigma u^{j_1}_{x_{k_1}}u^{ j_2}_{x_{k_2}}u^{j_3}_{x_{k_3}}dxds\Big |\notag,
\end{align}

where $\sigma(t)\equiv\min\{1,t\}$. We obtain {\em a priori} bounds for $\Phi(t)$ and $H(t)$ under the assumptions that
\begin{itemize}
\item the initial energy $C_0$ in \eqref{def of L^2 initial data} is small;
\item the density $\rho$ remains bounded above and non-negative.
\end{itemize}
The results can be summarised in the following theorem:

\begin{theorem}\label{a priori estimate with rho bounded theorem}  
Let $N$, $d$, $\bar\rho$, $\tilde\rho>0$ be given. Assume that the system parameters in \eqref{Navier Stokes Poisson} satisfy the conditions in \eqref{assumption on P}-\eqref{assumption on viscosity in q}. For $T>0$, if $(\rho-\tilde\rho,u,\phi)$ is the classical solution of \eqref{Navier Stokes Poisson}-\eqref{initial condition} defined on $\R^3\times[0,T]$ with smooth initial data $(\rho_0-\tilde\rho,u_0,\nabla\phi_0)\in H^4(\R^3)$ satisfying \eqref{compatibility condition} and \eqref{Lq bound on initial data}-\eqref{def of L^2 initial data}, and if
\begin{equation}\label{pointwise bound assumption on rho} 
\mbox{$0\le\rho(x,t)\le\bar\rho$ on $\R^3\times[0,T]$},
\end{equation}
then we can find positive constants $\delta_T$, $M_T$ and $\theta_T$ such that: if $(\rho_0-\tilde\rho,u_0,\nabla\phi_0)$ further satisfies  
\begin{equation}\label{smallness assumption a priori}
C_{0}\le\delta_T
\end{equation}
with 
\begin{align*}
0\le {\rm ess}\inf\rho_0 \le {\rm ess}\sup\rho_0<\bar\rho-d,
\end{align*}
then the following bound holds
\begin{equation}\label{bound on Phi and H with bounds on rho}
\mbox{$\Phi(t)+H(t)\le M_TC_{0}^{\theta_T}$ on $\R^3\times[0,T]$.}
\end{equation}
\end{theorem}

Unless otherwise specified, $M$ will denote a generic positive constant which depends on $T$ and the same quantities as the constant $C_T$ in the statement of Theorem~\ref{global existence theorem} but independent the regularity of initial data. And for simplicity, we write $P=P(\rho)$ and $\tilde P=P(\tilde\rho)$, etc., without further referring.

\medskip

We begin with the following $L^2$-estimate on $(\rho-\tilde\rho,u,\phi)$ which is valid for all $t\in[0,T]$.

\begin{lemma}\label{L^2 estimate lemma}
For all $t\in[0,T]$, we have
\begin{equation}\label{L^2 estimate}
\sup_{0\le s\le t}\intox\Big(\rho|u|^2+|\rho-\tilde\rho|^2+|\nabla\phi|^2\Big)+\intoxt|\nabla u|^2\le MC_0.
\end{equation}
\end{lemma}

\begin{proof}
By direct computation, we readily have
\begin{equation*}
\frac{1}{2}\intox\rho|u|^2\Big|_0^t+\intoxt\Big(\mu|\nabla u|^2+\lambda(\divv(u))^2\Big)+\intox\mathcal{G}(\rho)\Big|_0^t=\intoxt\rho\nabla\phi\cdot u,
\end{equation*}
where $\dis \int_{\R^d}\mathcal{G}(\rho)=\int_{\R^d}\left(\rho\int_{\tilde{\rho}}^{\rho}s^{-2}( P(s)-P(\tilde\rho))ds\right)$ is comparable to the $L^2(\R^3)$ norm of $(\rho-\tilde\rho)$ (see \cite{hoff95} for related discussion). On the other hand, using the equations \eqref{Navier Stokes Poisson}$_1$ and \eqref{Navier Stokes Poisson}$_3$, we have
\begin{equation*}
\Delta\phi_t=\rho_t=-\divv(\rho u),
\end{equation*}
and hence
\begin{equation*}
\frac{1}{2}\intox|\nabla\phi(x,t)|^2dx\Big|_0^t=-\intoxt\rho u\cdot\nabla\phi.
\end{equation*}
Therefore we obtain the following energy balance equation:
\begin{equation}\label{energy balance}
\frac{1}{2}\intox\rho|u|^2\Big|_0^t+\frac{1}{2}\intox|\nabla\phi|^2\Big|_0^t+\intoxt\Big(\mu|\nabla u|^2+\lambda(\divv(u))^2\Big)+\intox\mathcal{G}(\rho)\Big|_0^t=0,
\end{equation}
and the result \eqref{L^2 estimate} follows.
\end{proof}

Next we prove the following lemma which gives some auxiliary bounds on $\phi$. These bounds will be useful in later analysis.

\begin{lemma}\label{auxiliary estimate on phi lemma}
For all $t\in[0,T]$, we have
\begin{align}
\sup_{0\le s\le t}\intox|\nabla\phi_t(x,t)|^2dx&\le M\sup_{0\le s\le t}\intox\rho|u|^2(x,s)dx,\label{auxiliary estimate on phi 1}\\
\sup_{0\le s\le t}\|\nabla\phi(\cdot,s)\|_{L^\infty}&\le C(r)\Big(C_0+C_0^\frac{1}{r}\Big),\label{auxiliary estimate on phi 2}\\
\sup_{0\le s\le t}\|\Delta\phi(\cdot,s)\|_{L^\infty}&\le M\sup_{0\le s\le t}\|(\rho-\tilde\rho)(\cdot,s)\|_{L^\infty}\label{auxiliary estimate on phi 3}
\end{align}
where $r>3$ and $C(r)>0$ depends on $r$.
\end{lemma}

\begin{proof}
Since $\Delta\phi_t=\rho_t=-\divv(\rho u)$, so we have
\begin{equation*}
\intox\Delta\phi_t\cdot\phi_t(x,t)dx=\intox\rho u\cdot\nabla\phi_t(x,t)dx,
\end{equation*}
which implies
\begin{equation*}
\intox|\nabla\phi_t(x,t)|^2dx\le C\Big(\intox\rho|u|^2(x,t)dx\Big)^\frac{1}{2}\Big(\intox|\nabla\phi_t(x,t)|^2dx\Big)^\frac{1}{2},
\end{equation*}
and \eqref{auxiliary estimate on phi 1} follows by \eqref{L^2 estimate}. To prove \eqref{auxiliary estimate on phi 2}, using \eqref{sobolev bound r} and \eqref{L^2 estimate}, for $r>3$, there exists $\alpha_r\in(0,1)$ and $C(r)>0$ such that
\begin{align*}
\|\nabla\phi(\cdot,t)\|_{L^\infty}&\le C(r)\Big(\|\nabla\phi(\cdot,t)\|_{L^r}+\|\Delta\phi(\cdot,t)\|_{L^r}\Big)\\
&\le C(r)\Big(\|\nabla\phi(\cdot,t)\|_{L^2}^\frac{\alpha_r}{2}\|\Delta\phi(\cdot,t)\|_{L^2}^\frac{1-\alpha_r}{2}+M\|\Delta\phi(\cdot,t)\|^\frac{1}{r}_{L^2}\Big)\\
&\le C(r)\Big(C_0+C_0^\frac{1}{r}).
\end{align*}
Finally, \eqref{auxiliary estimate on phi 3} follows immediately from the equation \eqref{Navier Stokes Poisson}$_3$.
\end{proof}

We prove the following auxiliary $L^q$ estimates on the velocity $u$ which will be used for controlling the $L^3$ norm of $\nabla u$.

\begin{lemma}\label{Lq bound on u lemma} 
For all $t\in[0,T]$, we have
\begin{align}\label{Lq bound on u}
\sup_{0\le s\le t}\intox|u(t)|^q+\intoxt|u|^{q-2}|\nabla u|^2\le M\Big[C_0+\intox|u_0|^q\Big]\le M\left[C_0+N\right].
\end{align}
\end{lemma}
\begin{proof} 

The proof is similar to the one given in \cite{hoff05} page 323--324. From the momentum equation, we have
\begin{align*}
&\rho\left[(|u|^q)_t +(\nabla|u|^q)\cdot u\right]+q|u|^{q-2}u\cdot\nabla (P-\tilde P) \\
&+\mu q|u|^{q-2}|\nabla u|^2 +\lambda q|u|^{q-2}(\divv u)^2+|u|^q\rho\nabla\phi=q|u|^{q-2}\left[\frac{1}{2}\mu\Delta |u|^2 + \lambda\divv((\divv u)u)\right].
\end{align*}
Adding the equation $|u|^q (\rho_t + \divv(\rho u)) = 0$ and integrating, we then obtain, for $0\le t\le T$, 
\begin{align}\label{integral of Lq of u}
&\intox\rho|u|^{q}\Big|_0^t+\intoxt q|u|^{q-2}\Big[\mu|\nabla u|^2+\lambda(\divv u)^2+\mu(q-2)|\nabla|u||^2\Big]\notag\\
&\qquad+\intoxt q\lambda(\nabla|u|^{q-2})\cdot u\divv u+\intoxt |u|^q\rho\nabla\phi\notag\\
&=\Int\Big[q\divv(|u|^{q-2}u)(P-\tilde P)\Big].
\end{align}
Using the estimate \eqref{auxiliary estimate on phi 2}, for $r=4$ the term $\intoxt |u|^q\rho\nabla\phi$ can be bounded by 
$$\sup_{0\le s\le t}\|\rho\nabla\phi(\cdot,t)\|_{L^\infty}\intox|u(\cdot,t)|^q\le M(C_0+C_0^\frac{1}{4})\Big(\intox|u(\cdot,t)|^q\Big),$$ 
and using the hypothesis \eqref{assumption on viscosity in q} on $\mu$ and $\lambda$, we can bound the integrand in the double integral on the left side of \eqref{integral of Lq of u} from below as follows.
\begin{align*}
&q|u|^{q-2}\Big[\mu|\nabla u|^2+\lambda(\divv u)^2+\mu(q-2)|\nabla|u||^2-\lambda(q-2)|\nabla|u|||\divv u|\Big]\\
&\ge q|u|^{q-2}\Big[\mu(q-1)-\frac{1}{4}(q-2)^2\Big]|\nabla u|^2\\
&\ge M^{-1}|u|^{q-2}|\nabla u|^2.
\end{align*}
On the other hand, there exists some $M>0$ such that for each $\varepsilon>0$, the right side of \eqref{integral of Lq of u} is bounded by
\begin{align*}
&M\left[\frac{\varepsilon}{2}\Int|u|^{q-2}|\nabla u|^2+\frac{1}{2\varepsilon}\left(\int_{0}^{t}\!\!\!\int_{\R^3}|P-\tilde P|^{q}\right)^\frac{1}{2}\left(\int_0^t\!\!\!\int_{\R^3}|u|^{q}\right)^\frac{1}{2}\right]\\
&\le M\left[\frac{\varepsilon}{2}\Int|u|^{q-2}|\nabla u|^2+\frac{1}{2\varepsilon}\left(\int_{0}^{t}\!\!\!\int_{\R^3}|\rho-\tilde\rho|^{q}\right)^\frac{1}{2}\left(\int_0^t\!\!\!\int_{\R^3}|u|^{q}\right)^\frac{1}{2}\right].
\end{align*}
Hence we conclude that for all $\varepsilon>0$,
\begin{align*}
&\int_{\R^3}|u|^{q}\Big|_0^t+\Int|u|^{q-2}|\nabla u|^2\notag\\
&\le M\left[\frac{\varepsilon}{2}\Int|u|^{q-2}|\nabla u|^2+\frac{1}{2\varepsilon}\left(\int_{0}^{t}\!\!\!\int_{\R^3}|\rho-\tilde\rho|^{q}\right)^\frac{1}{2}\left(\int_0^t\!\!\!\int_{\R^3}|u|^{q}\right)^\frac{1}{2}\right]\notag\\
&\qquad+M(C_0+C_0^\frac{1}{4})\Big(\intox|u(\cdot,t)|^q\Big).
\end{align*}
By choosing $\varepsilon$ sufficient small and applying Gronwall's inequality, \eqref{Lq bound on u} follows for all $t\in[0,T]$.
\end{proof}

We recall the following bounds for $u,\omega$ and $ F$ in $W^{1,r}$ which are required for the derivation of estimates for the auxiliary functionals $H(t)$ and $\Phi(t)$. The proof can be found on page 505 in \cite{lm11}.

\begin{proposition}\label{Fourier multiplier lemma}  
For $r_1,r_2\in(1,\infty)$ and $t\in[0,T]$, we have the following estimates:
\begin{align}\label{Fourier multiplier on u}
||\nabla u(\cdot,t)||_{L^{r_1}}\le C(r_1)\Big[|| F(\cdot,t)||_{L^{r_1}}+||\omega(\cdot,t)||_{L^{r_1}}+||(\rho-\tilde\rho)(\cdot,t)||_{L^{r_1}}\Big]
\end{align}
and
\begin{align}\label{Fourier multiplier on F}
&||\nabla F(\cdot,t)||_{L^{r_2}}+||\nabla\omega(\cdot,t)||_{L^{r_2}}\notag\\
&\le C(r_2)\Big[||\rho^\frac{1}{2}\dot u(\cdot,t)||_{L^{r_2}}+||\nabla u(\cdot,t)||_{L^{r_2}}+\|(\rho-\tilde\rho)^2(\cdot,t)\|_{L^{r_2}}+\|\rho^\frac{1}{2}\nabla\phi(\cdot,t)||_{L^{r_2}}\Big].
\end{align}
The constants $C(r_1),C(r_2)$ in \eqref{Fourier multiplier on u}-\eqref{Fourier multiplier on F} may depend additionally on $r_1$ and $r_2$ respectively.
\end{proposition}

We are ready to prove some higher order estimates on $u$ and $\dot{u}$ which are crucial in bounding the functional $\Phi(t)$ in terms of $H(t)$ for all $t\in[0,T]$.

\begin{lemma}\label{H^1 and H^2 estimate lemma}
For all $t\in[0,T]$, we have
\begin{equation}\label{H^1 estimate}
\sup_{0\le s\le t}\sigma\intox|\nabla u|^2+\intoxt\sigma\rho|\dot{u}|^2\le M\Big(C_0+\Big |\sum_{1\le k_i,j_m\le 3}\int_{0}^{t}\!\!\!\int_{\R^3}\sigma u^{j_1}_{x_{k_1}}u^{ j_2}_{x_{k_2}}u^{j_3}_{x_{k_3}}dxds\Big |\Big),
\end{equation}
and
\begin{align}\label{H^2 estimate}
\sup_{0\le s\le t}\sigma^3\intox\rho|\dot{u}|^2+\intoxt\sigma^3|\nabla\dot{u}|^2\le M\Big(C_0^\theta+H(t)\Big)
\end{align}
for some $\theta>0$.
\end{lemma}

\begin{proof}
The proof is similar to the one given in \cite{hoff95}. To prove \eqref{H^1 estimate}, we first multiply the equation \eqref{Navier Stokes Poisson}$_2$ by $\sigma\dot{u}$ and integrate to get
\begin{align*}
&\sup_{0\le s\le t}\Big(\sigma\intox|\nabla u|^2\Big)+\intoxt\sigma\rho|\dot{u}|^2\\
&\le \Big|\intoxt\sigma\rho\dot{u}\cdot\nabla\phi\Big|+M\Big(C_0+\Big |\sum_{1\le k_i,j_m\le 3}\int_{0}^{t}\!\!\!\int_{\R^3}\sigma u^{j_1}_{x_{k_1}}u^{ j_2}_{x_{k_2}}u^{j_3}_{x_{k_3}}dxds\Big |\Big).
\end{align*}
For the term involving $\phi$, using H\"{o}lder's inequality and the estimate \eqref{L^2 estimate},
\begin{align*}
\Big|\intoxt\sigma\rho\dot{u}\cdot\nabla\phi\Big|&\le M\Big(\intoxt\sigma\rho|\dot{u}|^2\Big)^\frac{1}{2}\Big(\intoxt\sigma|\nabla\phi|^2\Big)^\frac{1}{2}\\
&\le M\Big(\intoxt\sigma\rho|\dot{u}|^2\Big)^\frac{1}{2}\Big(MC_0\Big)^\frac{1}{2},
\end{align*}
hence the term can be absorbed by the left side of \eqref{H^1 estimate}, and hence \eqref{H^1 estimate} follows.

Next, we apply the differential operator $\dt+u\cdot\nabla$ on the equation \eqref{Navier Stokes Poisson}$_2$ and make use of the transport theorem to obtain
\begin{align}\label{H^2 estimate 1}
&\sup_{0\le s\le t}\sigma^3\intox\rho|\dot{u}|^2+\intoxt\sigma^3|\nabla\dot{u}|^2\notag\\
&\le M\Big(C_0+\Big |\sum_{1\le k_i,j_m\le 3}\int_{0}^{t}\!\!\!\int_{\R^3}\sigma u^{j_1}_{x_{k_1}}u^{ j_2}_{x_{k_2}}u^{j_3}_{x_{k_3}}dxds\Big|+\intoxt\sigma^3|\nabla u|^4\Big)\notag\\
&\qquad+M\Big|\intoxt\sigma^3\dot{u}\cdot\frac{\partial}{\dt}(\rho\nabla\phi)\Big|+M\Big|\intoxt\sigma^3\dot{u}\cdot(u\cdot\nabla(\rho\nabla\phi))\Big|.
\end{align}
To estimate the term $M\Big|\intoxt\sigma^3\dot{u}\cdot\frac{\partial}{\dt}(\rho\nabla\phi)\Big|$, using equation \eqref{Navier Stokes Poisson}$_1$, we have
\begin{align*}
&\Big|\intoxt\sigma^3\dot{u}\cdot\frac{\partial}{\dt}(\rho\nabla\phi)\Big|\\
&=\Big|\intoxt\sigma^3\dot{u}\cdot(-\divv(\rho u)\nabla\phi+\rho\nabla\phi_t)\Big|\\
&\le M\intoxt\sigma^3\rho|\nabla\dot{u}||u||\nabla\phi|+M\intoxt\sigma^3\rho|\dot{u}||u||\Delta\phi|+M\intoxt\sigma^3\rho|\dot{u}||\nabla\phi_t|.
\end{align*}
Using the estimates \eqref{auxiliary estimate on phi 1}, \eqref{auxiliary estimate on phi 2} and \eqref{auxiliary estimate on phi 3}, we have
\begin{align*}
&\intoxt\sigma^3\rho|\nabla\dot{u}||u||\nabla\phi|\\
&\le M\Big(\intoxt\sigma^3|\nabla\dot{u}|^2\Big)^\frac{1}{2}\Big(\intoxt\rho|u|^2\Big)^\frac{1}{2}\Big(\sup_{0\le s\le t}\|\nabla\phi(\cdot,t)\|_{L^\infty}\Big)\\
&\le M\Big(\intoxt\sigma^3|\nabla\dot{u}|^2\Big)^\frac{1}{2}C_0^\frac{1}{2}\Big(C_0+C_0^\frac{1}{r}\Big),
\end{align*}
\begin{align*}
&\intoxt\sigma^3\rho|\dot{u}||u||\Delta\phi|\\
&\le M\Big(\intoxt\sigma\rho|\dot{u}|^2\Big)^\frac{1}{2}\Big(\intoxt\rho|u|^2\Big)^\frac{1}{2}\Big(\sup_{0\le s\le t}\|\Delta\phi(\cdot,t)\|_{L^\infty}\Big)\\
&\le M\Big(\intoxt\sigma^3|\nabla\dot{u}|^2\Big)^\frac{1}{2}C_0^\frac{1}{2},
\end{align*}
and
\begin{align*}
\intoxt\sigma^3\rho|\dot{u}||\nabla\phi_t|&\le M\Big(\intoxt\sigma\rho|\dot{u}|^2\Big)^\frac{1}{2}\Big(\intoxt|\nabla\phi_t|^2\Big)^\frac{1}{2}\\
&\le M\Big(\intoxt\sigma\rho|\dot{u}|^2\Big)^\frac{1}{2}C_0^\frac{1}{2},
\end{align*}
hence $\Big|\intoxt\sigma^3\dot{u}\cdot\frac{\partial}{\dt}(\rho\nabla\phi)\Big|$ can be readily bounded in terms of $\Phi$. For the term $M\Big|\intoxt\sigma^3\dot{u}\cdot(u\cdot\nabla(\rho\nabla\phi))\Big|$, we can bound it as follows.
\begin{align*}
&\Big|\intoxt\sigma^3\dot{u}\cdot(u\cdot\nabla(\rho\nabla\phi))\Big|\\
&\le M\intoxt\sigma^3\rho|\dot{u}||\nabla u||\nabla\phi|+M\intoxt\sigma^3\rho|\nabla\dot{u}||u||\nabla\phi|\\
&\le M\Big(\intoxt\sigma^3\rho|\dot{u}|^2\Big)^\frac{1}{2}\Big(\intoxt|\nabla u|^2\Big)^\frac{1}{2}\Big(\sup_{0\le s\le t}\|\nabla\phi(\cdot,t)\|_{L^\infty}\Big)\\
&\qquad+M\Big(\intoxt\sigma^3|\nabla\dot{u}|^2\Big)^\frac{1}{2}C_0^\frac{1}{2}\Big(C_0+C_0^\frac{1}{r}\Big)\\
&\le M\Phi^\frac{1}{2}C_0^\frac{1}{2}\Big(C_0+C_0^\frac{1}{r}\Big)+M\Big(\intoxt\sigma^3|\nabla\dot{u}|^2\Big)^\frac{1}{2}C_0^\frac{1}{2}\Big(C_0+C_0^\frac{1}{r}\Big).
\end{align*}
Therefore the result \eqref{H^2 estimate} follows.
\end{proof}

By \eqref{H^1 estimate} and \eqref{H^2 estimate} as obtained in Lemma~\ref{H^1 and H^2 estimate lemma}, we have the following bound on $\Phi(t)$ in terms of $H(t)$:

\begin{equation}\label{bound on Phi in terms of H}
\Phi(t)\le M\Big[C_0^{\theta_1}+C_0+H(t)\Big].
\end{equation}

In Lemma~\ref{bound on H lemma} listed below, we obtain the bound on $H(t)$ in terms of $\Phi(t)$.

\begin{lemma}\label{bound on H lemma} 
Assume that the hypotheses and notations of {\rm Theorem~2.1} are in force.  Then there are positive numbers $\theta_1>0$ and $\theta_2>1$ such that for all $t\in[0,T]$, we have
\begin{align}\label{bound on H}
H(t)\le M\Big[C_0^{\theta_1}+\Phi(t)^{\theta_2}\Big].
\end{align}
\end{lemma}
\begin{proof}
First we bound the term $\ds\int_0^t\sigma^3\|\nabla u(\cdot,s)\|_{L^4}^4$ as appeared in the definition of $H(t)$. With the help of \eqref{Fourier multiplier on u}, 
\begin{align}\label{L^4 bound on nabla u}
&\int_{0}^{t}\sigma^3\|\nabla u(\cdot,s)\|_{L^4}^4ds\notag\\
&\le M\int_{0}^{t}\sigma^3(\|F(\cdot,s)\|_{L^4}^4+\|\omega(\cdot,s)\|_{L^4}^4+\|(\rho-\tilde\rho)(\cdot,s)\|_{L^4}^4)ds.
\end{align}
We estimate $\dis\int_{0}^{t}\sigma^3\|F(\cdot,s)\|_{L^4}^4ds$. Using \eqref{sobolev bound r}, \eqref{Fourier multiplier on F} and the definition of $\Phi$, it can be estimated by
\begin{align*}
&\int_{0}^{t}\sigma^3\|F(\cdot,s)\|_{L^4}^4ds\\
&\le M\int_0^t\sigma^3\|F(\cdot,s)\|_{L^2}\|\nabla F(\cdot,s)\|_{L^2}^3ds\\
&\le M\left(\sup_{0\le s\le t}\sigma^\frac{1}{2}\|F(\cdot,s)\|_{L^2}\right)\left(\sup_{0\le s\le t}\sigma^\frac{3}{2}\|\nabla F(\cdot,s)\|_{L^2}\right)\int_0^t\sigma\|\nabla F(\cdot,s)\|_{L^2}^2ds\\
&\le M(\Phi(t)+C_0)^2.
\end{align*}
Using the same method, the terms $\dis\int_{0}^{t}\sigma^3\|\omega(\cdot,s)\|_{L^4}^4ds$ can be bounded by
\begin{align*}
\int_{0}^{t}\sigma^3\|\omega(\cdot,s)\|_{L^4}^4ds\le M(\Phi(t)+C_0)^2.
\end{align*}
Putting the estimates back into \eqref{L^4 bound on nabla u}, we obtain the following bound on $\ds\int_0^t\sigma^3\|\nabla u\|_{L^4}^4$:
\begin{align}\label{L^4 bound on nabla u final}
\int_{0}^{t}\sigma^3\|\nabla u(\cdot,s)\|_{L^4}^4ds\le M(\Phi(t)^2+C_0^2+C_0).
\end{align}
Next we estimate the term $\ds\int_0^t\sigma^\frac{3}{2}\|\nabla u(\cdot,s)\|_{L^3}^3$. Applying \eqref{sobolev bound r} and \eqref{Fourier multiplier on u}-\eqref{Fourier multiplier on F}, it can be bounded as follows.
\begin{align*}
&\int_{0}^{t}\sigma^\frac{3}{2}\|\nabla u(\cdot,s)\|_{L^3}^3ds\\
&\le M\int_0^t\sigma^\frac{3}{2}(\|F(\cdot,s)\|_{L^2}^\frac{3}{2}\|\nabla F(\cdot,s)\|_{L^2}^\frac{3}{2}+\|\omega\|_{L^2}^\frac{3}{2}\|\nabla \omega(\cdot,s)\|_{L^2}^\frac{3}{2})ds+MC_0.
\end{align*} 
The first term on the right can be estimated by
\begin{align*}
&\int_0^t\sigma^\frac{3}{2}\|F(\cdot,s)\|_{L^2}^\frac{3}{2}\|\nabla F(\cdot,s)\|_{L^2}^\frac{3}{2}ds\\
&\le M\left(\sup_{0\le s\le t}\sigma^\frac{1}{2}\|F(\cdot,s)\|_{L^2}^\frac{1}{2}\right)\left(\int_0^t\sigma\|F(\cdot,s)\|_{L^2}^2ds\right)^\frac{1}{4}\\
&\qquad\times\left(\int_0^t\sigma\|\nabla F(\cdot,s)\|_{L^2}^2ds\right)^\frac{3}{4}\\
&\le M(\Phi(t)+C_0)^\frac{1}{2}\left(\int_0^t\|\nabla u(\cdot,s)\|_{L^2}^2ds+\sup_{0\le s\le t}\|(\rho-\tilde\rho)(\cdot,s)\|_{L^2}^2\right)^\frac{1}{4}\\
&\qquad\times\left(\Phi(t)+\sup_{0\le s\le t}\|(\rho-\tilde\rho)(\cdot,s)\|_{L^4}^4\right)^\frac{3}{4}\\
&\le M(\Phi(t)+C_0)^\frac{3}{2}.
\end{align*}
Proceeding in the same way, the second term involving $\omega$ is also bounded by 
\begin{align*}
\int_0^t\sigma^\frac{3}{2}\|\omega(\cdot,s)\|_{L^2}^\frac{3}{2}\|\nabla \omega(\cdot,s)\|_{L^2}^\frac{3}{2}ds\le M(\Phi(t)+C_0)^\frac{3}{2}.
\end{align*}
Hence we obtain
\begin{align}\label{L^3 bound on nabla u final}
\int_0^t\sigma^\frac{3}{2}\|\nabla u(\cdot,s)\|_{L^3}^3ds\le M(\Phi(t)+C_0)^\frac{3}{2}.
\end{align}
It remains to estimate the summation term $\dis\sum_{1\le k_i,j_m\le 3}\Big|\int_{0}^{t}\!\!\!\int_{\R^3}\sigma u^{j_1}_{x_{k_1}}u^{ j_2}_{x_{k_2}}u^{j_3}_{x_{k_3}}\Big |$. The proof is similar to the one given in \cite{hoff95} page 239--241 and we just sketch here. We make use of the decomposition
\begin{align*}
u^j_{x_k x_k}&=(u^j_{x_k}-u^k_{x_j})_{x_k}+(u^k_{x_k})_{x_j}\\
&=\omega^{j,k}_{x_j}+(\divv u)_{x_j}\\
&=\omega^{j,k}_{x_j}+\left[(\mu+\lambda)^{-1}\tilde\rho F\right]_{x_j}+\left[(\mu+\lambda)^{-1}(P-\tilde P)\right]_{x_j},
\end{align*} 
and write $u$ as $u=z+w$ so that 
\begin{align*}
z^j_{x_k x_k}=\omega^{j,k}_{x_j}+\left[(\mu+\lambda)^{-1}\tilde\rho F\right]_{x_j}\mbox{\quad and\quad}w^j_{x_k x_k}=\left[(\mu+\lambda)^{-1}(P-\tilde P)\right]_{x_j}.
\end{align*}
Hence for each $r\in(1,\infty)$, there is a constant $M(r)$ such that for $t>0$,
\begin{align*}
\|\nabla z(\cdot,t)\|_{L^r}\le M(r)[\|F(\cdot,t)\|_{L^r}+\|\omega(\cdot,t)\|_{L^r}],
\end{align*}
\begin{align*}
\|\nabla w(\cdot,t)\|_{L^r}\le M(r)[\|(P-\tilde P)(\cdot,t)\|_{L^r}].
\end{align*}
Given $j_1,j_2,j_3,k_1,k_2,k_3\in\{1,2,3\}$, we have
\begin{align*}
\int_{0}^{t}\!\!\!\int_{\R^3}\sigma u^{j_1}_{x_{k_1}}u^{ j_2}_{x_{k_2}}u^{j_3}_{x_{k_3}}=A_1+A_2+A_3,
\end{align*}
where
\begin{align*}
|A_1|&\le\int_{0}^{t}\!\!\!\int_{\R^3}\sigma|\nabla u||\nabla z||\nabla w|,\\
|A_2|&\le\int_0^{1}\!\!\!\int_{\R^3}\sigma|\nabla u||\nabla w|^2,\\
A_3&=\int_{0}^{t}\!\!\!\int_{\R^3}\sigma u^{j_1}_{x_{k_1}}z^{ j_2}_{x_{k_2}}z^{j_3}_{x_{k_3}}.
\end{align*}
We readily have
\begin{align*}
\int_{0}^{t}\!\!\!\int_{\R^3}\sigma|\nabla u||\nabla z||\nabla w|&\le M\left(\int_{0}^{t}\sigma^\frac{3}{2}\|\nabla u(\cdot,t)\|_{L^3}^3\right)^\frac{2}{3}\left(\int_{0}^{t}\|\nabla w(\cdot,t)\|_{L^3}^3\right)^\frac{1}{3}\\
&\le M\Big[(\Phi(t)+C_0)^\frac{3}{2}+C_0^\frac{1}{2}(\Phi(t)^2+C_0^2+C_0)^\frac{1}{2}\Big]^\frac{2}{3}C_0^\frac{1}{3},
\end{align*}
and 
\begin{align*}
\int_0^{1}\!\!\!\int_{\R^3}\sigma|\nabla u||\nabla w|^2\le MC_0.
\end{align*}
For $A_3$, using \eqref{L^2 estimate} and \eqref{Lq bound on u lemma}, we can estimate it as follows.
\begin{align*}
\int_{0}^{t}\!\!\!\int_{\R^3}\sigma|u|^2|\nabla z|^2&\le \int_{0}^{t}\!\!\!\int_{\R^3}(|u|^2|\nabla u|^2+|u|^2|\nabla w|^2)\\
&\le\left(\int_{0}^{t}\|\nabla u(\cdot,s)\|_{L^2}^2\right)^\frac{q-4}{q-2}\left(\int_{0}^{t}\!\!\!\int_{\R^3}|u|^{q-2}|\nabla u|^2\right)^\frac{2}{q-2}\\
&\qquad+\int_{0}^{t}\|u(\cdot,s)\|_{L^6}^2\|\nabla w(\cdot,s)\|_{L^3}^2ds\\
&\le C_0^\frac{q-4}{q-2}(C_0+N)^\frac{2}{q-2}+MC_0C_0^\frac{1}{3}.
\end{align*}
Substituting the above estimates back into \eqref{L^4 bound on nabla u}, we obtain the bound \eqref{bound on H} as required.
\end{proof}

\begin{proof}[Proof of Theorem~\ref{a priori estimate with rho bounded theorem}]
Theorem~\ref{a priori estimate with rho bounded theorem} follows immediately from the bounds \eqref{bound on Phi in terms of H} and \eqref{bound on H}, and the fact that those functionals $\Phi(t)$ and $H(t)$ are all continuous in time.
\end{proof}

\section{Pointwise bounds for the density and Proof of Theorem~\ref{existence theorem}}\label{pointwise bound density section}

In this section we derive pointwise bounds for the density $\rho$, bounds which are independent both of time and of initial smoothness. This will then close the estimates of Theorem~\ref{a priori estimate with rho bounded theorem} to give an uncontingent estimate for the functionals $\Phi(t)$ and $H(t)$ defined in \eqref{def of Phi}-\eqref{def of H}. The result is formulated as Theorem~\ref{uncontingent estimate theorem}. Theorem~\ref{existence theorem} will then be proved using the {\it a priori} estimates derived from Theorem~\ref{uncontingent estimate theorem}.

\begin{theorem}\label{uncontingent estimate theorem}
Let $N$, $d$, $\bar\rho$, $\tilde\rho>0$ be given. Assume that the system parameters in \eqref{Navier Stokes Poisson} satisfy the conditions in \eqref{assumption on P}-\eqref{assumption on viscosity in q}. For $T>0$, if $(\rho-\tilde\rho,u,\phi)$ is the classical solution of \eqref{Navier Stokes Poisson}-\eqref{initial condition} defined on $\R^3\times[0,T]$ with initial data $(\rho_0-\tilde\rho,u_0,\nabla\phi_0)\in H^4(\R^3)$ satisfying \eqref{compatibility condition} and \eqref{Lq bound on initial data}-\eqref{def of L^2 initial data}, then we can find positive constants $\delta_T$, $M_T$ and $\theta_T$ such that: if $(\rho_0-\tilde\rho,u_0,\nabla\phi_0)$ further satisfies  
\begin{equation}\label{smallness assumption a priori uncontingent estimate}
C_{0}\le\delta_T
\end{equation}
with 
\begin{align}\label{boundedness on initial rho uncontingent estimate}
0\le {\rm ess}\inf\rho_0 \le {\rm ess}\sup\rho_0<\bar\rho-d,
\end{align}
then in fact
\begin{equation*}
\mbox{$\Phi(t)+H(t)\le M_TC_{0}^{\theta_T}$ on $\R^3\times[0,T]$,}
\end{equation*}
as well as
\begin{equation*} 
\mbox{$0\le\rho(x,t)\le\bar\rho$ on $\R^3\times[0,T]$}.
\end{equation*}
\end{theorem}

\begin{proof} The proof consists of a maximum-principle argument applied along particle trajectories of $u$ which is similar to the one given in \cite{sh12} pg. 48--50 for the corresponding magnetohydrodynamics system as well as those in \cite{suen13}-\cite{suen14} with slight modifications. We choose a positive number $b'$ satisfying
\begin{align*}
\bar\rho-d<b'<\bar\rho.
\end{align*}
Recall that $\rho_0$ takes values in $[0,\bar\rho-d)$, so that $\rho\in[0,\bar\rho]$ on $\R^3\times [0,\tau]$ for some positive $\tau$. It then follows from Theorem~\ref{a priori estimate with rho bounded theorem} that
\begin{align}\label{3.1'}
\Phi(\tau)+H(\tau)\le MC_0^\theta,
\end{align}
where $M$ is now fixed. We shall show that if $C_0$ is further restricted, then in fact $0\le \rho< b'$ on $\R^3\times [0,\tau]$, and so by an open-closed argument that $0\le\rho<b' $ on all of $\R^3\times[0,T]$, we have $\Phi(t)+H(t)\le MC_0^\theta$ as well. We will only give the proof for the upper bound, the proof of the lower
bound is just similar.

Fix $y\in\R^3$ and define the corresponding particle path $x(t)$ by
\begin{align*}
\left\{ \begin{array}
{lr} \dot{x}(t)
=u(x(t),t)\\ x(0)=y.
\end{array} \right.
\end{align*}
Suppose that there is a time $t_1\le\tau$ such that $\rho(x(t_1),t_1)=b'$. We may take $t_1$ minimal and then choose 
$t_0<t_1$ maximal such that $\rho(x(t_0),t_0) =\bar{\rho}$. Thus $\rho(x(t),t)\in [\bar{\rho}, b']$ for $t\in [t_0,t_1]$.

We have from the definition \eqref{def of effective viscous flux} of $F$ and the mass equation that
\begin{align*}
(\mu+\lambda)\frac{d}{dt}[\log\rho(x(t),t)]+P(\rho(x(t),t))-P(\tilde\rho)=-F(x(t),t).
\end{align*}
Integrating from $t_0$ to $t_1$ and and abbreviating $\rho(x(t),t)$ by $\rho(t)$, etc., we then obtain
\begin{equation}\label{integral on F}
(\mu+\lambda)[\log\rho(s)-\log(\tilde\rho)]\Big |_{ t_0}^{t_1}+\int_{ t_0}^{t_1}[P(s)-\tilde P]ds=-\int_{ t_0}^{t_1}F(s)ds.
\end{equation}
We shall show that 
\begin{equation}\label{smallness on the integral of F}
-\int_{ t_0}^{t_1}\tilde\rho(s)F(s)ds\le\tilde MC_0^{\theta'}
\end{equation}
for a constant $\tilde M$ which depends on the same quantities as the $M_T$ from Theorem~\ref{uncontingent estimate theorem} and $\theta'>0$.
If so, then from \eqref{integral on F}, 
\begin{equation}\label{smallness on the difference on density}
(\mu+\lambda)[\log\rho(s)-\log(\tilde\rho)]\Big |_{ t_0}^{t_1}\le\tilde MC_0^{\theta'},
\end{equation}
where the last inequality holds because $P$ is increasing and $P(s)-\tilde P(s)$ is nonnegative on $[t_0,t_1]$. But \eqref{smallness on the difference on density} cannot hold if $C_0$ is small depending on $\tilde M, b',$ and $\bar\rho$. Stipulating this smallness condition, we therefore conclude that there is no time $t_1$ such that $\rho(t_1) = \rho(x(t_1),t_1) = b'$. Since $y\in \R^3$ was arbitrary, it follows that $\rho<b'$ on $\R^3\times [0,\tau]$, as claimed.

To prove \eqref{smallness on the integral of F}, we rewrite the right hand side of \eqref{integral on F} as a {\it space-time} integral. Let $\Gamma$ be the fundamental solution of the Laplace operator in $\R^3$, then we apply \eqref{Poisson eqn of F} to obtain
\begin{align}
-\int_{ t_0}^{t_1}F(s)ds&=-\int_{ t_0}^{t_1}\!\!\!\int_{\R^3}\Gamma_{x_j}(x(s)-y)\rho\dot{u}^{j}(y,s)dyds\notag\\
&\qquad+\int_{ t_0}^{t_1}\!\!\!\int_{\R^3}\Gamma_{x_j}(x(s)-y)\Big[\rho\phi_{x_j}(y,s)\Big]dyds.\label{integral on F using Green fct}
\end{align}
Using \eqref{bound on green function}, the first integral on the right of \eqref{integral on F using Green fct} is bounded in the same way as in Lemma~4.2 of Hoff~\cite{hoff95}
\begin{align*}
\int_{ t_0}^{t_1}\!\!\!\int_{\R^3}&\Gamma_{x_j}(x(t)-y)\rho\dot{u}^{j}(y,t)dyds\\
&\le||\Gamma_{x_j}*(\rho u^{j})(\cdot,t_1)||_{L^{\infty}}+||\Gamma_{x_j}*(\rho u^{j})(\cdot, t_0)||_{L^{\infty}}\\
&\qquad+\left |\int_{ t_0}^{t_1}\!\!\!\int_{\R^3}\Gamma_{x_j x_k}(x(s)-y)\left[u^{k}((x(s),s)-u^{k}(y,s)\right](\rho u^{j})(y,s)dyds\right |\\
&\le \tilde MC_0^{\theta'}.
\end{align*}
For the second integral on the right side of \eqref{integral on F using Green fct}, we observe that, for all $s\in[0,\tau]$,
\begin{align*}
\int_{\R^3}\Gamma_{x_j}&(x(s)-y)\Big[\rho\phi_{x_j}(y,s)\Big]dy\notag\\
&\le\tilde M\int_{\R^3}|y|^{-2}|\nabla\phi(y,s)| dy\notag\\
&\le\tilde M\left[\int_{|y|\le1}|y|^{-\frac{5}{2}}dy\right]^\frac{4}{5}\left[\int_{|y|\le1}|\nabla\phi(y,s)|^5 dy\right]^\frac{1}{5}\notag\\
&\qquad\qquad+\tilde M\left[\int_{|y|>1}|y|^{-4}dy\right]^\frac{1}{2}\left[\int_{|y|>1}|\nabla\phi(y,s)|^2 dy\right]^\frac{1}{2}\notag\\
&\le\tilde M\Big(\intox|\nabla\phi(y,s)dy\Big)^\frac{1}{20}\Big(\intox|\Delta\phi(y,s)dy\Big)^\frac{9}{20}+\tilde MC_0^\frac{1}{2}\\
&\le\tilde MC_0^\frac{1}{2}
\end{align*}
where we have used \eqref{sobolev bound r} for $r=5$. Therefore we finish the proof.
\end{proof}

\begin{proof}[Proof of Theorem~\ref{existence theorem}]
Let $T>0$ be given. By Theorem~\ref{local existence theorem}, for a given initial data $(\rho_0-\tilde\rho,u_0,\nabla\phi_0)\in H^4(\R^3)$, we can find $\bar{t}>0$ such that the classical solution $(\rho-\tilde\rho,u,\phi)$ to \eqref{Navier Stokes Poisson}-\eqref{initial condition} exists on $\R^3\times[0,\bar{t}]$. Furthermore, by Theorem~\ref{uncontingent estimate theorem} just proved, there exists $\delta_{\bar{t}},M_{\bar{t}},\theta_{\bar{t}}>0$ such that if $(\rho_0-\tilde\rho,u_0,\nabla\phi_0)$ satisfies the bounds \eqref{compatibility condition} and \eqref{Lq bound on initial data}-\eqref{def of L^2 initial data} and \eqref{boundedness on initial rho uncontingent estimate} and the smallness assumption 
\begin{equation}\label{smallness assumption proof}
C_0<\delta_{\bar{t}}, 
\end{equation}
then $(\rho-\tilde\rho,u,\phi)$ satisfies
\begin{align*}
0\le\rho(x,t)\le\bar\rho,
\end{align*}
and
\begin{align*}
\Phi(t)+H(t)\le M_{\bar{t}}C_{0}^{\theta_{\bar{t}}}.
\end{align*}
for all $(x,t)\in\R^3\times[0,\bar{t}]$. In particular, we have
\[\rho(x,\bar{t})\in[0,\bar\rho],\]
and $(\rho-\tilde\rho,u,\nabla\phi)(\cdot,\bar{t})\in H^4(\R^3)$. Since the system \eqref{Navier Stokes Poisson} is autonomous, we can therefore reapply Theorem~\ref{local existence theorem} at the new initial time $\bar{t}$ to extend the solution $(\rho-\tilde\rho,u,\phi)$ eventually to all $[0,T]$ in finitely many steps by replacing $\delta_{\bar{t}}$ with some smaller number $\delta_T$ in \eqref{smallness assumption proof}. In other words, there exists $\delta_T>0$ depending on $T$ such that if the initial data $(\rho_0-\tilde\rho,u_0,\nabla\phi_0)\in H^4(\R^3)$ is given satisfying \eqref{compatibility condition} and \eqref{Lq bound on initial data}-\eqref{def of L^2 initial data} with $C_0<\delta_T$, then $(\rho-\tilde\rho,u,\phi)$ can be extended to the whole time interval $[0,T]$. It finishes the proof of Theorem~\ref{existence theorem}.
\end{proof}

\section{Existence of weak solutions and Proof of Theorem~\ref{global existence theorem}}\label{weak solution section}

In this section, we prove Theorem~\ref{global existence theorem} by obtaining weak solutions to the system \eqref{Navier Stokes Poisson}-\eqref{initial condition} on $\R^3\times[0,T]$ where $T>0$ is any given time. To begin with, we let initial data $(\rho_0,u_0,\phi_0)$ be given satisfying the hypotheses \eqref{compatibility condition} and \eqref{Lq bound on initial data}-\eqref{def of L^2 initial data} of Theorem~\ref{uncontingent estimate theorem}, and we fix those constants $\delta_T$, $M_T$ and $\theta_T$ defined in Theorems~\ref{uncontingent estimate theorem}. 

Upon choosing $(\rho^\eta_0,u^\eta_0,\phi_0^\eta)$ as a smooth approximation of $(\rho_0,u_0,\phi_0)$ (which can be obtained by convolving $(\rho_0,u_0,\phi_0)$ with the standard mollifying kernel of width $\eta>0$), we can apply Theorem~\ref{existence theorem} to show that there is a smooth solution $(\rho^\eta,u^\eta,\phi^\eta)$ of \eqref{Navier Stokes Poisson}-\eqref{initial condition} with initial data $(\rho^\eta_0,u^\eta_0,\phi_0^\eta)$ defined on the time interval $[0,T]$. The {\em a priori} estimates of Theorem~\ref{uncontingent estimate theorem} then apply to show that
\begin{equation}\label{estimates on smooth approx solution}
\mbox{$\Phi(t)+H(t)\le M_TC_{0}^{\theta_T}$ on $\R^3\times[0,T]$,}
\end{equation}
as well as
\begin{equation}\label{bound on smooth rho} 
\mbox{$0\le\rho(x,t)\le\bar\rho$ on $\R^3\times[0,T]$}.
\end{equation}
where $\Phi(t)$ and $H(t)$ are defined by \eqref{def of Phi}-\eqref{def of H} but with $(\rho,u,\phi)$ replaced by $(\rho^\eta,u^\eta,\phi^\eta)$.

We prove that the approximate solution $(\rho^\eta,u^\eta,\phi^\eta)$ satisfies the following uniform H\"older continuity estimates. It will be used for obtaining weak solutions via compactness arguments.

\begin{lemma}\label{holder continuity lemma}
Given $\tau\in(0,T]$, there is a constant $C=C(\tau)$ such that for all $\alpha\in(0,\frac{1}{2}]$ and $\eta>0$,
\begin{align}\label{Holder continuity in space}
\langle u^\eta(\cdot,t)\rangle^{\alpha}&\le M\Big[\Big(C_0+||\nabla u^\eta(\cdot,t)||^2_{L^2}\Big)^{\frac{1-2\alpha}{4}}\Big(||(\rho^\eta)^\frac{1}{2}\dot{u}^\eta(\cdot,t)||^2_{L^2}+||\rho^\frac{1}{2}\nabla\phi^\eta(\cdot,t)||^2_{L^2}\Big)^{\frac{1+2\alpha}{4}}\\
&\qquad+||\nabla u^\eta(\cdot,t)||_{L^2}^\frac{1-2\alpha}{2}||\nabla\omega^\eta(\cdot,t)||_{L^2}^\frac{1+2\alpha}{2} + C_0^{\frac{1-\alpha}{3}}\Big].\notag\\
&\le C(\tau)C_0^{\theta_T},\notag
\end{align}
\begin{align}\label{Holder continuity in space-time}
\langle u^\eta(\cdot,t)\rangle^{\frac{1}{2},\frac{1}{4}}_{\R^3\times [\tau,T]}\le C(\tau)C_0^{\theta_T}.
\end{align}
\end{lemma}
\begin{proof} First we prove \eqref{Holder continuity in space}. Let $\alpha\in (0,\frac{1}{2}]$ and define $r\in (3,6]$ by $r=3/(1-\alpha)$. Then by \eqref{holder bound r} and \eqref{Fourier multiplier on u},
\begin{align}\label{holder estimate on u}
\langle u^\eta(\cdot,t)\rangle^{\alpha}&\le M\|u^\eta(\cdot,t)\|_{L^r}\notag\\
&\le M\left[||F^\eta(\cdot,t)||_{L^r}+||\omega^\eta(\cdot,t)||_{L^r}+||(\rho^\eta-\tilde\rho)(\cdot,t)||_{L^r}\right].
\end{align}
But by \eqref{sobolev bound r}, we also have
\begin{align*}
\|F^\eta(\cdot,t)\|_{L^r} &\le M\left(\|F^\eta(\cdot,t)\|_{L^2}^{(6-r)/2r}\|\nabla F^\eta(\cdot,t)\|_{L^2}^{(3r-6)/2r}\right)\\
&\le M\left(||(\rho^\eta-\tilde\rho)(\cdot,t)||^2_{L^2}+||\nabla u^\eta(\cdot,t)||^2_{L^2}\right)^\frac{1-2\alpha}{4}\\
&\qquad\times\left(||(\rho^\eta)^\frac{1}{2}\dot{u}^\eta(\cdot,t)||^2_{L^2}+||(\rho^\eta)^\frac{1}{2}\nabla\phi^\eta(\cdot,t)||^2_{L^2}\right)^\frac{1+2\alpha}{4}
\end{align*}
and
\begin{align*}
\|\omega^\eta(\cdot,t)\|_{L^r} &\le M\left(\|\omega^\eta(\cdot,t)\|_{L^2(\R^3)}^{(6-r)/2r}\|\nabla\omega^\eta(\cdot,t)\|_{L^2(\R^3)}^{(3r-6)/2r}\right)\\
&\le M\left(||\nabla u^\eta(\cdot,t)||_{L^2}^\frac{1-2\alpha}{2}||\nabla\omega^\eta(\cdot,t)||_{L^2}^\frac{1+2\alpha}{2}\right),
\end{align*}
where $F^\eta$ is defined in \eqref{def of effective viscous flux} with $u,\rho$ replaced by $u^\eta,\rho^\eta$ respectively. The result \eqref{Holder continuity in space} then follows by applying these bounds in \eqref{holder estimate on u} as well as the uniform estimates \eqref{estimates on smooth approx solution}.

Next we consider \eqref{Holder continuity in space-time}. First, the H\"older-$\frac{1}{2}$ continuity of $u^\eta$ in space was just proved in \eqref{Holder continuity in space} by taking $\alpha=\frac{1}{2}$, and to infer H\"older continuity in time we fix $x$ and $\tau\le t_1\le t_2\le T$ and compute
\begin{align}\label{5.1.02}
|u^\eta (x,t_2)-u^\eta (x,t_1)|&\le\frac{1}{|B_{R}(x)|}\int_{B_R(x)}\left |u^{\eta}(z,t_2)-u^{\eta}(z,t_1)\right |dz + C(\tau)C_0^{\theta_T}R^{\frac{1}{2}}\notag\\
&\le R^{-\frac{3}{2}}|t_2 - t_1|\sup_{\tau\le t\le T}\left( \int |u^{\eta}_t|^2dx\right)^{\frac{1}{2}} + C(\tau)C_0^{\theta_T}R^{\frac{1}{2}}\notag\\
&\le C(\tau)C_0^{\theta_T}\left[R^{-\frac{3}{2}}|t_2 - t_1|+R^{\frac{1}{2}}\right]
\end{align}
by the bounds in \eqref{estimates on smooth approx solution}. Taking $R=|t_2 - t_1|^\frac{1}{2}$ we then obtain the estimate in \eqref{Holder continuity in space-time} for $u^\eta$.
\end{proof}

In view of Lemma~\ref{holder continuity lemma}, the desired weak solution $(\rho,u,\phi)$ will then be obtained by taking $\eta\to 0$, with the use of the compactness provided by those bounds in \eqref{estimates on smooth approx solution}-\eqref{bound on smooth rho} and \eqref{Holder continuity in space-time}. It can be summarised in the following lemma:

\begin{lemma}\label{compactness arguments lemma}
There is a sequence $\eta_k\to 0$ and functions $u$, $\rho$ and $\phi$ such that as $k\to\infty$,
\begin{align}\label{strong convergence on compact set}
\mbox{ $u^{\eta_k}\rightarrow u$ uniformly on compact sets in $\R^3\times (0,T)$};\end{align}
\begin{align}\label{weak convergence in L2 space}
\nabla u^{\eta_k}(\cdot,t),\nabla\omega^{\eta_k}(\cdot,t)\rightharpoonup\nabla u(\cdot,t),\nabla\omega(\cdot,t)
\end{align}
weakly in $L^2(\R^3)$
for all $t\in(0,T]$; 
\begin{align}\label{weak convergence in L2 space-time}
&\mbox{$\sigma^{\frac{1}{2}}\dot{u}^{\eta_k},\sigma^{\frac{3}{2}}\nabla\dot{u}^{\eta_k}\rightharpoonup\sigma^{\frac{1}{2}}\dot{u},\sigma^{\frac{3}{2}}\nabla\dot{u}$}
\end{align}
weakly in $L^2(\R^3\times[0,T])$; and
\begin{align}\label{strong convergence L2loc}
\rho^{\eta_k}(\cdot,t),\,\Delta\phi^{\eta_k}\to \rho(\cdot,t),\,\Delta\phi(\cdot,t)
\end{align}
strongly in $L^2_{loc}(\R^3)$ for every $t\in[0,T]$.
\end{lemma}
\begin{proof} 
The uniform convergence \eqref{strong convergence on compact set} follows from Lemma~\ref{holder continuity lemma} via a diagonal process, thus fixing the sequence $\{\eta_k\}$. The weak-convergence statements in \eqref{weak convergence in L2 space} and \eqref{weak convergence in L2 space-time} then follow from the bound \eqref{estimates on smooth approx solution} and considerations based on the equality of weak-$L^2$ derivatives and distribution derivatives. The convergence of approximate densities in \eqref{strong convergence L2loc} for a further subsequence can be achieved by applying the argument given in Lions \cite{lions98} pp. 21--23 and extended by Feireisl \cite{feireisl04}, pp. 63--64 and 118--127, and we omit the details here.
\end{proof}

\begin{proof}[{\bf Proof of Theorem~\ref{global existence theorem}:}]
In view of the bounds \eqref{estimates on smooth approx solution} and \eqref{Holder continuity in space-time} and the convergence results obtained from Lemma~\ref{compactness arguments lemma}, it is clear that the limiting functions $(\rho,u,\phi)$ of Lemma~\ref{compactness arguments lemma} inherit the bounds in \eqref{L2 of nabla u weak solution}-\eqref{enegry bound on weak solution}. It is also clear from the modes of convergence described in Lemma~\ref{compactness arguments lemma} that $(\rho,u,\phi)$ satisfies the weak forms \eqref{weak form of mass equation}-\eqref{weak form of poisson equation} of the system \eqref{Navier Stokes Poisson} as well as the initial condition \eqref{initial condition}. The continuity statement \eqref{1.4.1} then follows easily from these weak forms together with the bounds in \eqref{enegry bound on weak solution}. This completes the proof of Theorem~\ref{global existence theorem}.
\end{proof}

\section{Blow-up criteria and Proof of Theorem~\ref{blow up theorem}}\label{blow-up section}

In this section, we study the blow-up criterion for classical solutions to the system \eqref{Navier Stokes Poisson}-\eqref{initial condition} and prove Theorem~\ref{blow up theorem}. First we define the so-called maximal time of existence of smooth solutions to \eqref{Navier Stokes Poisson}-\eqref{initial condition}:
\begin{definition}\label{def of maximal time}
We call $T^*\in(0, \infty)$ to be the maximal time of existence of a smooth solution $(\rho-\tilde\rho,u,B)$ to \eqref{Navier Stokes Poisson}-\eqref{initial condition} if for any $0<T<T^*$, $(\rho-\tilde\rho,u,\phi)$ solves \eqref{Navier Stokes Poisson}-\eqref{initial condition} in $[0, T]\times\R^3$ and satisfies 
\begin{equation}\label{smoothness on rho}
\rho - \tilde{\rho}\in C^0([0,T];H^4(\R^3))\cap C^1 ([0,T];H^3(\R^3)),
\end{equation}
\begin{equation}\label{smoothness on u}
u\in C^0([0,T];H^4(\R^3))\cap C^1([0,T];H^2 (\R^3)),
\end{equation}
and
\begin{equation}\label{smoothness on phi}
\nabla\phi\in C^0([0,T];H^5(\R^3)).
\end{equation} 
Moreover, the conditions \eqref{smoothness on rho}-\eqref{smoothness on phi} fail to hold when $T=T^*$.
\end{definition}
We will prove Theorem~\ref{blow up theorem} using a contradiction argument. Specifically, for the sake of contradiction, we assume that
\begin{equation}\label{bound on rho assumption for contradiction}
||\rho||_{L^\infty((0,T^*)\times\R^3)}\le\bar{C}.
\end{equation} 
for some constant $\bar{C}>0$. Based on the assumption \eqref{bound on rho assumption for contradiction}, we derive {\em a priori} estimates for the local smooth solution $(\rho-\tilde\rho,u,\phi)$ on $[0,T]$ with $T\le T^*$. Those estimates are different from we did in Section~\ref{Energy Estimates section} and Section~\ref{pointwise bound density section} in the sense that there is no smallness assumption imposed on the initial data, hence a more delicate analysis is required in bounding the solution $(\rho-\tilde\rho,u,\phi)$.

To facilitate the proof, we introduce the following auxiliary functionals:
\begin{align}
\Phi_1(t)&=\sup_{0\le s\le t}\|\rho^\frac{1}{2}\dot{u}(\cdot,s)\|^2_{L^2}+\int_{0}^{t}\|\rho^\frac{1}{2}\dot{u}(\cdot,s)\|_{L^2}^2ds,\label{def of Phi 1}\\
\Phi_2(t)&=\sup_{0\le s\le t}\|\rho^\frac{1}{2}\dot{u}(\cdot,s)\|^2_{L^2}+\int_{0}^{t}\|\nabla\dot{u}(\cdot,s)\|^2_{L^2}ds,\label{def of Phi 2}\\
\Phi_3(t)&=\int_0^t\!\!\!\int_{\R^3}|\nabla u|^4dxds.\label{def of Phi 3}
\end{align}

We first recall the following lemma which gives estimates on the solutions of the Lam\'{e} operator $\mu\Delta+(\mu+\lambda)\nabla\divv$. More detailed discussions can also be found in Sun-Wang-Zhang \cite{swz11}.

\begin{lemma}\label{estimates on Lame operator}
Consider the following equation:
\begin{equation}\label{eqn for Lame operator}
\mu\Delta v+(\mu+\lambda)\nabla\divv(v)=J,
\end{equation}
where $v=(v^1,v^2,v^3)(x)$, $J=(J^1,J^2,J^3)(x)$ with $x\in\R^3$ and $\mu$, $\lambda>0$. Then for $r\in(1,\infty)$, we have:
\begin{enumerate} 
\item if $J\in W^{2,r}(\R^3)$, then $||\Delta v||_{L^r}\le \tilde C||J||_{L^r}$;
\item if $J=\nabla\varphi$ with $\varphi\in W^{2,r}(\R^3)$, then $||\nabla v||_{L^r}\le \tilde C||\varphi||_{L^r}$;
\item if $J=\nabla\divv(\varphi)$ with $\varphi\in W^{2,r}(\R^3)$, then $||v||_{L^r}\le \tilde C||\varphi||_{L^r}$.
\end{enumerate}
Here $\tilde C$ is a positive constant which depends only on $\mu$, $\lambda$ and $r$.
\end{lemma}
\begin{proof}
A proof can be found in \cite{swz11} pg. 39 and we omit the details here.
\end{proof}

Given $\tilde\rho>0$ and initial data $(\rho_0-\tilde\rho,u_0,\nabla\phi_0)\in H^4(\R^3)$, we define
\begin{equation}
S_0=\|(\rho_0-\tilde\rho,u_0,\nabla\phi_0)\|_{H^4}^2.
\end{equation}
We begin to estimate the functionals $\Phi_1$, $\Phi_2$ and $\Phi_3$ under the assumption \eqref{bound on rho assumption for contradiction}. Similar to the previous cases, $M$ will denote a generic positive constant which further depends on $\bar{C}$, $\tilde C$, $T^*$ and $S_0$. 

We first have the following bounds based on the results we obtained in Section~\ref{Energy Estimates section}, namely for any $0\le t\le T\le T^*$ and $r>3$,
\begin{align}
\sup_{0\le s\le t}\intox\Big(\rho|u|^2+|\rho-\tilde\rho|^2+|\nabla\phi|^2\Big)+\intoxt|\nabla u|^2\le M,\label{L^2 estimate blow up}\\
\sup_{0\le s\le t}\intox(|u|^4+|B|^4)\le M,\label{L^4 estimate blow up}\\
\sup_{0\le s\le t}\intox|\nabla\phi_t(x,t)|^2dx\le M\sup_{0\le s\le t}\intox\rho|u|^2(x,s)dx,\label{auxiliary estimate on phi 1 blow up}\\
\sup_{0\le s\le t}\|\nabla\phi(\cdot,s)\|_{L^\infty}\le C(r)\Big(S_0+S_0^\frac{1}{r}\Big),\label{auxiliary estimate on phi 2 blow up}\\
\sup_{0\le s\le t}\|\Delta\phi(\cdot,s)\|_{L^\infty}\le M\sup_{0\le s\le t}\|(\rho-\tilde\rho)(\cdot,s)\|_{L^\infty}.\label{auxiliary estimate on phi 3 blow up}
\end{align}
Next we are going to estimate $\Phi_1$ which is given in the following lemma:
\begin{lemma}\label{bound on Phi 1 lemma}
For any $0\le t\le T\le T^*$,
\begin{equation}\label{bound on Phi 1}
\Phi_1(t)\le M[1+\Phi_3(t)].
\end{equation}
\end{lemma}
\begin{proof}
Following the proof of Lemma~\ref{H^1 and H^2 estimate lemma}, we have
\begin{align}\label{H^1 estimate blow up}
\sup_{0\le s\le t}\intox|\nabla u|^2+\intoxt\rho|\dot{u}|^2\le \Big|\intoxt\rho\dot{u}\cdot\nabla\phi\Big|+M\Big(S_0+\intoxt|\nabla u|^3\Big).
\end{align}
Using \eqref{bound on rho assumption for contradiction} and \eqref{L^2 estimate blow up}, the first term on the right side of \eqref{H^1 estimate blow up} can be estimated as follows.
\[\begin{aligned}
\Big|\intoxt\rho\dot{u}\cdot\nabla\phi\Big|&\le \bar{C}\Big(\intoxt\rho|\dot{u}|^2\Big)^\frac{1}{2}\Big(\intoxt|\nabla\phi|^2\Big)^\frac{1}{2}\\
&\le\bar{C}\Phi_1^\frac{1}{2}T^\frac{1}{2}(MS_0)^\frac{1}{2}\le M\Phi_1^\frac{1}{2},
\end{aligned}\]
and the second term can be bounded by
\[\begin{aligned}
\intoxt|\nabla u|^3&\le \Big(\intoxt|\nabla u|^2\Big)^\frac{1}{2}\Big(\intoxt|\nabla u|^4\Big)^\frac{1}{2}\\
&\le M\Phi_3^\frac{1}{2}.
\end{aligned}\]
Applying the above bounds on \eqref{H^1 estimate blow up}, the result follows.
\end{proof}
Before we estimate $\Phi_2$, we introduce the following decomposition on $u$. We write
\begin{equation}\label{decomposition on u}
u=u_p+u_s,
\end{equation}
where $u_p$ and $u_s$ satisfy
\begin{align}\label{def of u_p and u_s} 
\left\{ \begin{array}{l}
\mu\Delta(u_p)+(\mu+\lambda)\nabla\divv(u_p)=\nabla(P-\tilde P),\\
\rho(u_s)_t-\mu\Delta u_s-(\mu+\lambda)\nabla\divv(u_s)=-\rho u\cdot\nabla u-\rho(u_p)_t+\rho\nabla\phi.
\end{array}\right.
\end{align}

Using Lemma~\ref{estimates on Lame operator}, the term $u_p$ can be bounded by
\begin{equation}\label{estimate on u_p}
\intox|\nabla u_p|^r\le M\intox|P-\tilde P|^r\le M\intox|\rho-\tilde\rho|^r.
\end{equation}

We give the estimates for $u_s$ as follows.
\begin{lemma}\label{estimate on u_s lemma}
For any $0\le t\le T\le T^*$, we have
\begin{equation}\label{estimate on u_s}
\sup_{0\le \tau\le t}\intox|\nabla u_s|^2+\intoxt\rho|\dt(u_s)|^2+\intoxt|\Delta u_s|^2\le M.
\end{equation}
\end{lemma}
\begin{proof}
We multiply \eqref{def of u_p and u_s}$_2$ by $\dt(u_s)$ and integrate to obtain
\begin{align}\label{estimate on u_s step 1}
&\intox\mu|\nabla u_s|^2\Big|_0^t+\intoxt(\mu+\lambda)|\divv u_s|^2+\intoxt\rho|\dt(u_s)|^2\\
&=-\intoxt\Big(\rho u\cdot\nabla u\Big)\cdot\dt(u_s)-\intoxt\Big(\rho\dt(u_p)\Big)\cdot\dt(u_s)+\intoxt\rho\nabla\phi\cdot\dt(u_s).\notag
\end{align}
We estimate the right side of \eqref{estimate on u_s step 1} term by term. Using \eqref{L^4 estimate blow up} and \eqref{estimate on u_p}, the first integral can be bounded by
\begin{align*}
&\Big(\intoxt|u|^2|\nabla u|^2\Big)^\frac{1}{2}\Big(\intoxt\rho|\dt(u_s)|^2\Big)^\frac{1}{2}\\
&\le M\Big[\int_0^t\Big(\intox|u|^4\Big)^\frac{1}{2}\Big(\intox|\nabla u_s|^4+\intox|\nabla u_p|^4\Big)^\frac{1}{2}\Big]^\frac{1}{2}\Big(\intoxt\rho|\dt(u_s)|^2\Big)^\frac{1}{2}\\
&\le M\Big[\int_0^t\Big(\intox|\nabla u_s|^2\Big)^\frac{1}{4}\Big(\intox|\Delta u_s|^2\Big)^\frac{3}{4}+\int_0^t\Big(\intox|\rho-\tilde\rho|^4\Big)^\frac{1}{2}\Big]^\frac{1}{2}\\
&\qquad\qquad\qquad\times\Big(\intoxt\rho|\dt(u_s)|^2\Big)^\frac{1}{2}\\
&\le M\Big(\intoxt\rho|\dt(u_s)|^2\Big)^\frac{1}{2}\Big[\Big(\intoxt|\nabla u_s|^2\Big)^\frac{1}{8}\Big(\intoxt|\Delta u_s|^2\Big)^\frac{3}{8}+1\Big].
\end{align*}
Next to estimate $-\intoxt\Big(\rho\dt(u_p)\Big)\cdot\dt(u_s)$, we differentiate \eqref{def of u_p and u_s}$_1$ with respect to $t$ and obtain
\begin{equation*}
\mu\Delta\dt(u_p)+(\mu+\lambda)\nabla\divv\dt(u_p)=\nabla\divv(-P\cdot u).
\end{equation*}
Using Lemma~\ref{estimates on Lame operator} and \eqref{L^2 estimate blow up}, we have
\begin{align}\label{estimate on dt u_p}
\intoxt|\dt(u_p)|^2\le M\intoxt|P\cdot u|^2\le M.
\end{align}
Therefore
\begin{align*}
-\intoxt\Big(\rho\dt(u_p)\Big)\cdot\dt(u_s)&\le\Big(\intoxt\rho|\dt(u_s)|^2\Big)^\frac{1}{2}\Big(\intoxt|\dt(u_p)|^2\Big)^\frac{1}{2}\\
&\le M\Big(\intoxt\rho|\dt(u_s)|^2\Big)^\frac{1}{2}.
\end{align*}
To estimate $\intoxt\rho\nabla\phi\cdot\dt(u_s)$, using \eqref{L^2 estimate blow up}, we readily have
\begin{align*}
\intoxt\rho\nabla\phi\cdot\dt(u_s)&\le \Big(\intoxt\rho|\nabla\phi|^2\Big)^\frac{1}{2}\Big(\intoxt\rho|\dt(u_s)|^2\Big)^\frac{1}{2}\\
&\le M\Big(\intoxt\rho|\dt(u_s)|^2\Big)^\frac{1}{2}.
\end{align*}
Combining the above, we have from \eqref{estimate on u_s step 1} that 
\begin{align}\label{estimate on u_s step 2}
&\intox|\nabla u_s|^2(x,t)dx+\intoxt|\divv(u_s)|^2+\intoxt\rho|\dt(u_s)|^2\notag\\
&\le M\Big(\intoxt|\nabla u_s|^2\Big)^\frac{1}{4}\Big(\intoxt|\Delta u_s|^2\Big)^\frac{3}{4}+M.
\end{align}
It remains to estimate the term $\intoxt|\Delta u_s|^2$. Rearranging the terms in \eqref{def of u_p and u_s}$_2$, we have that
\begin{equation*}
\mu\Delta u_s+(\mu+\lambda)\nabla\divv(u_s)=\rho\dt(u_s)+\rho u\cdot\nabla u+\rho\dt(u_p)-\rho\nabla\phi.
\end{equation*}
Therefore, we can apply Lemma~\ref{estimates on Lame operator} to get
\begin{align*}
&\intoxt|\Delta u_s|^2\\
&\le M\Big[\intoxt(|\rho\dt(u_s)|^2+|\rho u\cdot\nabla u|^2+|\rho\dt(u_p)|^2+|\rho\nabla\phi|^2)\Big]. 
\end{align*}
Using \eqref{L^2 estimate blow up} and \eqref{estimate on dt u_p}, we obtain
\begin{equation*}
\intoxt(|\rho\dt(u_s)|^2+|\rho u\cdot\nabla u|^2+|\rho\dt(u_p)|^2+|\rho\nabla\phi|^2)\le M\Big(\intoxt\rho|\dt(u_s)|^2+1\Big),
\end{equation*}
and hence
\begin{equation}\label{estimate on Delta u_s}
\intoxt|\Delta u_s|^2\le M\Big(\intoxt\rho|\dt(u_s)|^2+1\Big).
\end{equation}
Applying the estimate \eqref{estimate on Delta u_s} on \eqref{estimate on u_s step 2} and using Gr\"{o}wall's inequality, we conclude that for $0\le t\le T\le T^*$,
\begin{equation*}
\intox|\nabla u_s|^2(x,t)dx\le M,
\end{equation*}
and the result \eqref{estimate on u_s} follows.
\end{proof}

We are now ready to estimate $\Phi_2$ as defined in \eqref{def of Phi 2}. The result is given in the following lemma.

\begin{lemma}\label{bound on Phi 2 lemma}
For any $0\le t\le T\le T^*$,
\begin{equation}\label{bound on Phi 2}
\Phi_2(t)\le M[\Phi_1(t)+\Phi_3(t)+1].
\end{equation}
\end{lemma}

\begin{proof}
Following the proof of Lemma~\ref{H^1 and H^2 estimate lemma}, we have
\begin{align}\label{bound on Phi 2 step 1}
&\intox|\dot{u}|^2+\intoxt|\nabla\dot{u}|^2\notag\\
&\le M\Big[S_0+\Big|\intoxt\dot{u}\cdot\frac{\partial}{\dt}(\rho\nabla\phi)\Big|+\Big|\intoxt\dot{u}\cdot(u\nabla(\rho\nabla\phi))\Big|+\Phi_1+\Phi_3\Big].
\end{align}
Using the estimates \eqref{L^2 estimate blow up}-\eqref{auxiliary estimate on phi 3 blow up}, we readily have the bound
\begin{align*}
&\Big|\intoxt\dot{u}\cdot\frac{\partial}{\dt}(\rho\nabla\phi)\Big|\\
&\le M\intoxt\sigma^3\rho|\nabla\dot{u}||u||\nabla\phi|+M\intoxt\sigma^3\rho|\dot{u}||u||\Delta\phi|+M\intoxt\sigma^3\rho|\dot{u}||\nabla\phi_t|\\
&\le M\Big[\Big(\intoxt|\nabla\dot{u}|^2\Big)^\frac{1}{2}+\intoxt\rho|\dot{u}|^2\Big)^\frac{1}{2}\Big]\le M[\Phi_1^\frac{1}{2}+\Phi_2^\frac{1}{2}],
\end{align*}
as well as
\begin{align*}
&\Big|\intoxt\dot{u}\cdot(u\nabla(\rho\nabla\phi))\Big|\\
&\le M\Big[\intoxt\rho|\dot{u}||\nabla u||\nabla\phi|+\intoxt\rho|\nabla\dot{u}||u||\nabla\phi|\Big]\\
&\le M\Big[\Big(\intoxt|\nabla\dot{u}|^2\Big)^\frac{1}{2}+\intoxt\rho|\dot{u}|^2\Big)^\frac{1}{2}\Big]\le M[\Phi_1^\frac{1}{2}+\Phi_2^\frac{1}{2}].
\end{align*}
Therefore by Cauchy's inequality, the result \eqref{bound on Phi 2} follows.
\end{proof}

We finally obtain the bound on $\Phi_3$ in terms of $\Phi_2$.

\begin{lemma}\label{bound on Phi 3 lemma}
For any $0\le t\le T\le T^*$,
\begin{equation}\label{bound on Phi 3}
\Phi_3(t)\le M[\Phi_1(t)^\frac{1}{2}+1].
\end{equation}
\end{lemma}

\begin{proof}
Using the decomposition \eqref{decomposition on u} on $u$ and the estimates \eqref{estimate on u_p} and \eqref{estimate on u_s}, we have
\begin{align*}
\Phi_3&\le \intoxt|\nabla u_s|^4+\intoxt|\nabla u_p|^4\\
&\le M\int_0^t\Big(\intox|\nabla u_s|^2\Big)^\frac{1}{2}\Big(\intox|\Delta u_s|^2\Big)^\frac{3}{2}+\intoxt|\rho-\tilde\rho|^4\\
&\le M\Big[\Big(\sup_{0\le s\le t}\intox|\Delta u_s|^2\Big)^\frac{1}{2}+1\Big].
\end{align*}
To estimate $\intox|\Delta u_s|^2$, we rearrange the terms in \eqref{def of u_p and u_s}$_2$ to obtain
\[\mu\Delta u_s+(\mu+\lambda)\nabla\divv(u_s)=\rho\dot{u}-\rho\nabla\phi.\]
Therefore Lemma~\ref{estimates on Lame operator} implies that
\begin{align*}
\intox|\Delta u_s|^2\le M\Big[\intox(|\rho\dot{u}|^2+|\rho\nabla\phi|^2)\Big]\le M(\Phi_2+1),
\end{align*}
and the result follows.
\end{proof}

\begin{proof}[Proof of Theorem~\ref{blow up theorem}]
In view of the bounds \eqref{bound on Phi 1}, \eqref{bound on Phi 2} and \eqref{bound on Phi 3}, we can conclude that for $0\le t\le T\le T^*$,
\begin{equation}\label{bound on Phi 1 Phi 2 Phi 3}
\Phi_1(t)+\Phi_2(t)+\Phi_3(t)\le M.
\end{equation}
Together with the pointwise boundedness assumption \eqref{bound on rho assumption for contradiction} on $\rho$ and apply the similar argument given in \cite{suen13a}, we can show that for $0\le t\le T\le T^*$, 
\begin{align}\label{H^4 bounds blow up}
\sup_{0\le s\le t}||(\rho-\tilde{\rho},u,\nabla\phi)(\cdot,s)||_{H^4} + \int_{0}^{t}||u(\cdot,s)||^2_{H^4}ds\le M'',
\end{align}
for some $M''$ which depends on $C_0$, $\bar{C}$, $T^*$ and the system parameters $P$, $\mu$, $\lambda$ and $K$. An open-and-closed argument on the time interval can then be applied which shows that the local solution $(\rho-\tilde\rho,u,\phi)$ can be extended beyond $T^*$, which contradicts the maximality of $T^*$. Therefore the assumption \eqref{bound on rho assumption for contradiction} does not hold and this completes the proof of Theorem~\ref{blow up theorem}.
\end{proof}



\end{document}